\documentclass{amsart}
\usepackage{amsmath}
  \usepackage{paralist}
  \usepackage{graphics} 
  \usepackage{epsfig} 
 \usepackage[colorlinks=true]{hyperref}
\hypersetup{urlcolor=blue, citecolor=red}

\theoremstyle{definition}
\newtheorem{theor}{Theorem}[section]
\newtheorem{coro}{Corollary}
\newtheorem{lemma}[theor]{Lemma}
\newtheorem{prop}{Proposition}
\theoremstyle{definition}

\newtheorem{remark}{Remark}

\newcommand{\refq}[1]{~(\ref{#1})}

\newcommand{\egal}{\Longleftrightarrow}
\newcommand{\sauf}{\setminus}

\newcommand{\epsl}{\varepsilon} 
\newcommand{\bfbeta}{\boldsymbol{\beta}}
\newcommand{\bfepsl}{\boldsymbol{\epsl}} 
\newcommand{\bfmu}{\boldsymbol{\mu}} 
\newcommand{\bfalpha}{\boldsymbol{\alpha}}

\newcommand{\calA}{\mathcal A}
\newcommand{\calB}{\mathcal B}
\newcommand{\calC}{\mathcal C}
\newcommand{\calD}{\mathcal D}
\newcommand{\calE}{\mathcal E}
\newcommand{\calF}{\mathcal F}

\newcommand{\calH}{\mathcal H}
\newcommand{\calHD}{\calH^D}
\newcommand{\calK}{\mathcal K}
\newcommand{\calL}{\mathcal L}
\newcommand{\calO}{\mathcal O}

\newcommand{\calT}{\mathcal T}
\newcommand{\calU}{\mathcal U}

\newcommand{\calX}{\mathcal X}
\newcommand{\calZ}{\mathcal Z}

\newcommand{\rmd}{{\rm d}}
\newcommand{\sgn}{{\rm sgn}}
\newcommand{\supp}{{\rm supp}\,}

\newcommand{\Om}{\Omega}
\newcommand{\om}{\omega}
\newcommand{\ove}{\overline}

\newcommand{\C}{\mathbb C}

\newcommand{\N}{\mathbb N}
\newcommand{\R}{\mathbb R}
\newcommand{\Z}{{\rm Z\!\!\! Z}}
\newcommand{\Sm}{\mathbb S}
\newcommand{\T}{\mathbb T}

\title[Spectral analysis of the discrete Maxwell operator]{Spectral analysis of the discrete Maxwell operator: the Limiting Absorption Principle}
\author[Olivier Poisson]{}

\subjclass{Primary: 35Q61,47A10 ; Secondary: 47B25 }
\keywords{Maxwell operator, spectral analysis, limiting absorption, threshold, Mourre estimate.}

\email{olivier.poisson@univ-amu.fr}

\thanks{The author is grateful to Hiroshi Isozaki (University of Kyoto) for the suggestion of this work.}

\date{\today}

\begin{document}
\maketitle
\medskip
\centerline{\scshape Olivier Poisson }
{\footnotesize
 \centerline{ Aix-Marseille Universit\'{e}, France}
} 

\begin{abstract}
 We are interested by the spectral analysis of the anisotropic discrete Maxwell operator $\hat H^D$
 defined on the square lattice $\Z^3$.
 In aim to prove that the limiting absorption principle holds we construct a conjugate operator
 to the Fourier series of $\hat H^D$ at any not-zero real value.
 In addition we show that at some particular thresholds the conjugate operator is essentially 
 self-adjoint.
\end{abstract}

\pagestyle{myheadings}
\thispagestyle{plain}

\section{Introduction}
 We are interested in the limiting absorption principle (LAP) for the inhomogeneous discrete Maxwell
 operator $\hat H^D$ defined on the square lattice $\Z^3$, i.e.
 the limit in some usual or more abstract Besov spaces of the resolvents $\hat R(z)=(\hat H^D-z)^{-1}$
 when $z\in\C\sauf\R$ tends to a spectral value $\lambda$ of $\hat H^D$.
 It is equivalent and more convenient to deal with the Fourier series $H^D$ of $\hat H^D$
 instead of $\hat H^D$ since $H^D$ is a multiplication operator on the three-dimensional
 real torus $\T^3\approx (\R/(2\pi\Z))^3$.
 The operator $H^D$ is so represented by the analytically fibered self-adjoint operator
 $\T^3 \ni x\mapsto H^D(x)$ on an Hilbert space $\calHD$ where $H^D(x)$ is a $6\times 6$
 real matrix.
 We prove that the LAP holds by using the conjugate operator technique, as developed in
 the greatest generality by Gerard and Nier in a first version \cite{GER.MOU}.
 Actually, denoting $\Sigma=\{(\lambda,x)\,| \; \lambda\in \sigma(H^D(x))\}$ the energy-momentum set,
 where $\sigma(H^D(x))$ denotes the spectrum of an $H^D(x)$, we have the stratification
 $\Sigma= \cup_{j=1}^6 \Sigma_j$ where $\Sigma_i$ is the semi-analytical set of elements
 $(\lambda,x)$ for which $\lambda$ is an eigenvalue of multiplicity $i$ of $H^D(x)$ so
 we can introduce the (finite) set $\calT$ of thresholds.
 At this point we could simply apply \cite[Theorem 3.1]{GER.MOU} and claim that for any interval
 $I \subset\subset \R\sauf \calT$, there exists an operator $A_I$ with domain
 $\calC^\infty(\T^3;\C^6)$, essentially self-adjoint  on $\calHD$, conjugated to $H^D$, 
 and satisfying Points (i)-(iii) of \cite[Theorem 3.1]{GER.MOU}. 
 See Points (i)-(iv) of Theorem \ref{th.1}.
 However we go beyond this result. Moreover, the construction of the conjugate operator
 in \cite{GER.MOU} contains an error since, in fact, the analytically fibered self-adjoint operator
 for which the LAP is proved is not of class $\calC^2(A_I)$.
 The authors of \cite{GER.MOU} have recently corrected it in the new version \cite{GER.MOU2}
 at the cost of losing the (strict) globality of Mourre's inequality\refq{ineq.Mourre1}
 which has to be replaced be a local one as\refq{ineq.Mourre3}.
 Nevertheless, following the ideas of the old version \cite{GER.MOU} and the new version
 \cite{GER.MOU2} we construct an explicit conjugate operator $A_\phi$ to $H^D$ with
 the same properties than the operator $A_I$ of \cite[Theorem 3.1]{GER.MOU}.
 Here we prefer to use as parameter for the conjugate a numerical smooth function $\phi$
 instead of an interval $I$ so $A_\phi$ is conjugated to $H^D$ at each point of $\phi^{-1}(\R^*)$.
 
 We prove that the LAP holds at any nonzero threshold.
 Consequently, $H^D$ has no other eigenvalue than 0.
 This result is optimal since 0 is the obvious eigenvalue of $H^D$ so there is no conjugate
 operator to $H^D$ at 0. 
 Denoting by $\calX^*$ the finite set of points $x\in\T^3$ such that $\sigma(H^D(x))\cap\calT$
 is not reduced to the eigenvalue 0, we construct $A_\phi$ as a first order symmetric differential operator
 with smooth coefficients outside $\calX^*$ and with rational singularities at points
 $x\in\calX^*$ such that $\sigma(H^D(x)) \cap \calT \cap \phi^{-1}(\R^*)$ is not void.
 Then, the set of thresholds admits the partition $\calT=\calT_{sa} \cup \calT_{sm}$:
 let $x^*\in\calX^*$ and $\lambda\in \calT\cap \sigma(H^D(x^*))$, if $\lambda\in \calT_{sa}$
 then the local part $A_{x^*}$ of $A_\phi$ near $x^*$ is essentially self-adjoint, 
 and if $\lambda\in \calT_{sm}$ then $A_\phi$ near $x^*$ admits a maximal monotone extension.
 Thus, we obtain the validity of the LAP on $\R^*\sauf\calT_{sm}$ in the same terms as
 \cite[Theorem 3.3]{GER.MOU}.\\
 If $\lambda\in \phi^{-1}(\R^*)\cap \calT_{sm}$ then $A_\phi$ may not have a maximal monotone
 extension since its singularities may come from several points of $\calX^*$ and since
 the sum of two maximal monotone operators, even with disjoint supports, is not necessarily
 maximal monotone. Nevertheless we prove that the LAP holds on $\calT_{sm}$ by a slight extension
 of \cite[Theorem 3.3]{GEO.COM}.
 
\medskip

 In Part \ref{sec.2} we describe the operator $H^D$ and the Hilbert space $\calHD$
 on which it is self-adjoint.
 In Part \ref{sec.3} we state the main results of our work, notably, the existence of
 a conjugate operator $A_\phi$ for $H^D$ at Theorem \ref{th.1}.
 Then, we state that the LAP holds in many situations in terms of abstract or usual Besov spaces
 at Corollaries \ref{coro.1}-\ref{coro.5}.
 In Part \ref{sec.4} we describe the spectrum of $H^D(x)$, the stratification of
 the energy-momentum set $\Sigma$; we define the thresholds and describe $\calT$.
 It appears that $\calT$ depends mainly on the parameter $\bfbeta=\bfepsl\times\bfmu$
 where $\bfepsl$ and $\bfmu$ denote respectively the permittivity and the permeability.
 In fact outside the first quite basic case $\bfbeta=0$ we have to deal with many special cases
 for which the set $\calT$ is no more obvious.
 In Part \ref{sec.5} we construct the conjugate operator $A_\phi$ according to 
 the parameter $\bfbeta$. We prove also the main properties of $A_\phi$.
 In Part \ref{sec.6} we prove the main results (see Part \ref{sec.3}).
 
\medskip
 All along the text we use the following notations.
 Let $\calE\in\{\T^3,\R^3\}$ where $\T^3\approx (\R/(2\pi\Z))^3$ is the $3$-dimensional real torus.
 If $f$ is a numerical function (from $\R$ into $\R$)  we then denote by $f$ again the mapping
 $\calE \ni x \mapsto (f(x_1),f(x_2),f(x_3))\in \R^3$, so if $E\subset \calE$ then
 $f(E)=\{f(x);\; x\in E\}\subset \R^3$ and if $E\subset \R^3$ then $f^{-1}(E)=\{x\in\calE ;\; f(x)\in E\}$.
 In particular we set, for $x\in \calE$,
\begin{eqnarray*}
 y &=& \sin x :=(y_1=\sin x_1,y_2=\sin x_2,y_3=\sin x_3) \in \R^3, \\
 z &=& y^2 :=(z_1=y_1^2,z_2=y_2^2,z_3=y_3^2) \in \R^3.
\end{eqnarray*}
 (Although the letter $z$ will be also used as a complex energy, there is no possible confusion
 with the notation above.) The other notations are standard.
 If $E\subset \calE\in\{\T^3,\R^3\}$ and $\calF\in \{\R^n,\C^n\}$, we denote by $\calC^\infty_c(E,\calF)$
 the real space of $\calC^\infty$ functions with values in $\calF$, defined on $\calE$
 and with compact support in $E$.\\
 Let $T$ be a self-adjoint operator. We denote by $\sigma(T)$ the spectrum of $T$; 
 if $E\subset\R$ then $1_E(T)$  operator $\chi_{[a,b)}(T)$ where $\chi_J$ is the characteristic
 function of a set $J\subset\R$.\\
 If $X$ and $Y$ are two metrics spaces $B(X,Y)$ is the space of bounded operators from $X$
 into $Y$ and $B(X):=B(X,X)$. \\
%
 For $n\ge 1$ the space $\R^n$ (respect., $\C^n$) is equipped with the scalar product
 $<\cdot,\cdot>_{\R^n}$ (respect., with the hermitian product
$$
 <f,g>_{\C^n}  =   \sum_{j=1}^n f_j\ove{g_j}, \quad f=(f_j)_{1\le j \le n},\: g=(g_j)_{1\le j \le n}).
$$
 Full notations are at the end of the text, Part \ref{sec.notation}.

\section{The discrete Maxwell Operator}
\label{sec.2}
\subsection{The homogeneous discrete Maxwell Operator}
 Let $\Z^3=\{n=(n_1,n_2,n_3);\; n_j\in\Z\}$ the square lattice, $\T^3\approx (\R/(2\pi\Z))^3$
 the $3$-dimensional real torus, $U$ the unitary transform between $L^2(\T^3)$ and $l^2(\Z^3)$:
$$ (Uf)(n) := \hat f (n) \equiv (2\pi)^{-\frac32} \int_{\T^3} e^{inx} f(x)\rmd x  , \quad n\in\Z^3,$$
 so that any $f\in L^2(\T^3)$ can be written
$$
 f(x) = (U^* \hat f)(x) \equiv (2\pi)^{-\frac32}\sum_{n\in\Z^3} e^{-inx} \hat f(n), \quad x\in\T^3.
$$
 The homogeneous discrete Maxwell operator is the bounded operator $H_0$ on
 $\calH=(L^2(\T^3))^6$ defined by
$$ \hat H_0 = UH_0U^* , $$
 with $H_0(x)$ the real anti-symmetrical $6\times 6$ matrix:
$$ H_0(x) = \left(\begin{array}{cc} 0_{3\times3} & M(y)\\ -M(y) & 0_{3\times3} \end{array}\right) \in \R^6,  \quad x\in\T^3, $$ 
%
 where $y=\sin x$ and $M(y)$ is the real anti-symmetrical $3\times 3$ matrix:
\begin{equation}
\label{def.matM}
 M(y) = \left(\begin{array}{ccc}
  0 & -y_3 & y_2\\ y_3 &0 & -y_1\\ -y_2 & y_1 & 0 \end{array}\right) , \quad y\in \R^3.
\end{equation}
 The space $\calH=L^2(\T^3,\rmd x;\R^6)$ can be written as the hilbertian sum
$$ \calH =  \int_{\T^3}^{\oplus} \R^6 \rmd x,$$
 with the scalar product
$$ (u,v) :=  \int_{\T^3} <u(x),v(x)>_{\R^6} \rmd x.$$
 We then have
 $$ H_0 = \int_{\T^3}^{\oplus} H_0(x) \rmd x$$
 which shows that $H_0$ is a bounded self-adjoint operator on $\calH$.
 
 All along the article we consider that the vectorial functions $u$ and $v$ may be complex-valued
 so we denote by the same symbol $\calH$ the hilbertian space $\int_{\T^3}^{\oplus} \C^6 \rmd x$
 equipped with the hilbertian product
$$ (u,v) :=  \int_{\T^3} <u(x),v(x)>_{\C^6} \rmd x =  \int_{\T^3} \sum_{j=1}^6 u_j(x)\ove{v_j(x)} \rmd x$$
 and with norm $\|u\|:=(u,u)^{\frac12}$.
 
\subsection{The discrete Maxwell Operator}
 Let  $\bfepsl$ and $\bfmu$ be $3\times 3$ constant diagonal matrices with diagonal elements,
 respectively, $\epsl_1, \epsl_2, \epsl_3\in (0,\infty)$ for $\bfepsl$, and $\mu_1,\mu_2,\mu_3
 \in (0,\infty)$ for $\bfmu$.
 The inhomogeneous discrete-Maxwell operator is defined by
$$ \hat H^D = \hat D \hat H_0, $$
  where we set
$$
 \hat D = \left(\begin{array}{cc} \bfepsl & 0_{3\times3} \\ 0_{3\times3} &\bfmu
  \end{array}\right).
$$
 The (unperturbed) discrete Maxwell operator is $H^D=U^*\hat H^D U$ so we have,
 since $\hat D$ is constant,
$$
 H^D = U^*(\hat D \hat H_0)U = \hat D U^* \hat H_0 U = \hat D H_0.
$$
 The relation
$$
  (\hat D^{-1} H^D u,v) \, =\,  (H_0 u,v)
$$
 shows that the operator $H^D$ is bounded and self-adjoint on the space $\calHD=\calH$
 equipped with the hilbertian product
$$ (u,v)_{\calHD} := (\hat D^{-1}u,v) = \int_{\T^3} <\hat D^{-1}u(x),v(x)>_{\C^6} \rmd x $$
 and with norm $\|u\|_{\calHD}:=(u,u)_{\calHD}^{\frac12}$.
 (Since the norms $\|\cdot\|$ and $\|\cdot\|_{\calHD}$ are equivalent, we can omit the index
 "$ _{\calHD}$".)
 It appears that $H^D$ is a multiplication operator and we write
$$ H^D = \int_{\T^3}^{\oplus} H^D(x) \rmd x, $$
 where $H^D(x)$ is self-adjoint on $\C^6$ equipped with the hermitian product:
$$
   <u(x),v(x)>_{\C^6,D} := <\hat D^{-1}u(x),v(x)>_{\C^6}.
$$
 Since $H^D(x)$ depends only on $y=\sin(x)$ we write it $H^D(x)=h^D(y)$.
 It is known (see \cite{REE.MET}) that
\begin{equation}
\label{val1.sigma}
 \sigma(H^D) =  \ove{\cup_{x\in\T^3} \sigma(H^D(x))},
\end{equation}
 which is a compact set of $\R$.
\section{Results}
\label{sec.3}
\subsection{Main Theorem}
 We introduce the following settings:
\begin{eqnarray*}
 \calX_0 &:=& \{x\in \T^3;\, z=0\} , \quad  \T_0^3 = \T^3 \sauf \calX_0,\\
 \calX^* &:=& \{x\in\T^3; \: \sigma(H^D(x))\cap\calT\neq \{0\}\} \subset \T_0^3.
 \end{eqnarray*}
 We denote $\calD_0:=\calC_c^\infty(\T^3\sauf\calX^*)$ and $R(z):=(H^D-z)^{-1}$ the resolvent of $H^D$.
%
\begin{theor}
\label{th.1}
 Let $\calT'\subset\calT$ with $0\in\calT'$, and $\phi\in \calC^\infty_c(\R\sauf\calT')$.
 Then there exists a symmetric operator $A_{\phi}$ defined on $\calC^\infty(\T^3)$ if $\calT'= \calT$
 and on $\calD_0$ if $\calT'\neq \calT$, and satisfying the following properties:
\begin{itemize}
 \item[(i)] 
 There exists a constant $\delta=\delta(\phi)>0$ so that we have
\begin{equation}
\label{ineq.Mourre1}
\phi(H^D)[H^D,iA_{\phi}]\phi(H^D) \ge \delta \phi^2(H^D).
\end{equation}
 \item[(ii)] The multi-commutators $ad^k_{A_\phi}(H^D)$ are bounded for all $k\in \N$.
 \item[(iii)] The operator $A_{\phi}$ is a first order differential operator in $x$ whose coefficients
 belong to $\calC^\infty(\T^3;\calL(\C^6))$ if $\calT'= \calT$ and to 
 $\calC^\infty(\T^3\sauf\calX^*;\calL(\C^6))$ if $\calT' \neq \calT$,
 and there exists $\tilde\phi\in \calC_c^\infty(\R\sauf\calT')$ so that
 $A_{\phi}=\tilde\phi(H^D)A_{\phi}=A_{\phi}\tilde\phi(H^D)$.
 \item[(iv)] If $\calT_{sm}\subset\calT' $ then $A_\phi$ is essentially self-adjoint.
 \item[(v)] If $\calT'=\{0\}$ then $A_{\phi}$ has the form $\sum_{j=0}^3 A_j$ where
 $A_0$ has smooth coefficients and is essentially self-adjoint,
 $A_1$ and $A_2$ have coefficients with rational singularities at some points of $\calX^*$
 and, defined on the domain $\calD_0$, admit a maximal symmetric extension;
 moreover, $\supp A_1\cap \supp A_2=\emptyset$.
\end{itemize}
\end{theor}
\begin{remark}
%
\begin{itemize}
\item Since the coefficients of $A_\phi$ are smooth in $\T^3\sauf\calX^*$ and,
 since $\calD_0$ is dense in $\calHD$, then, for $u\in \calHD$, the distribution $A_\phi u$
 belongs to the topological dual space of $\calD_0$ so, since  $A_\phi$ is symmetric,
$$ A_\phi u\in \calHD \egal  |(u,A_\phi v)_{\calHD}| \le C\|v\| \quad \forall v\in \calD_0.$$
 Setting 
$$ \|u\|_{D(A_\phi)}:= \|u\|+ \|A_\phi u\| \quad \forall u \in \calD_0,$$
 the closure $\bar A_\phi$ of $A_\phi$ has domain
$$ D(\bar A_\phi):= \ove{\calD_0}^{\|\cdot\|_{D(A_\phi)}} .$$
 The adjoint of $\bar A_\phi$ or of $A_\phi$ is the operator $A_\phi \!^*$ with domain
 $D(A_\phi \!^*)=\{v\in \calHD$; $|(A_\phi u,v)_{\calHD}| \le C\|u\|$, $u\in D(A_\phi)\}$.
 If $\supp\phi\cap\calT=\emptyset$ then $A_\phi$ is essentially self-adjoint
 and its self-adjoint extension has domain $\ove{\calC^\infty(\T^3)}^{\|\cdot\|_{D(A_\phi)}}
  \supset H^1(\T^3)$ (with dense inclusion).

\item 
 Our set $\calT$ of thresholds may defer to the set of thresholds in \cite{GER.MOU}.
 Anyway, $\calT$ is a discrete (and finite) set.
 The sets $\calT$, $\calT_{sa}$ and $\calT_{sm}$ are precisely described in Part \ref{sec.taub}.
\item In the case $\calT'=\calT$, the result of \cite{GER.MOU} implies the existence of
 an essentially self-adjoint operator $A_I$ with smooth coefficients such that
 Points (i), (iii) and (iv) with $A_\phi$ replaced by $A_I$ hold.
 But the first commutator $[H^D,A_I]$ is not a multiplication operator so Point (ii) fails,
 and, in fact, $H^D\not\in \calC^{1,1}(A_I)$ (this set is defined at Point B
 of Corollary \ref{coro.1}).
 The new version \cite{GER.MOU2} of \cite{GER.MOU} provides an essentially self-adjoint
 operator $A_{I,I_1}$ with smooth coefficients such that Points (ii)$-$(iv) and a local (weaker)
 version of Point (i) are maintained (with $A_\phi$ replaced by $A_{I,I_1}$).
\item We give an explicite formula for $A_\phi$ which is easier to read than the general formula
 in \cite{GER.MOU2} (which is only valid in the case $\calT'=\calT$).
\end{itemize}
\end{remark}
\subsection{Main consequences and extensions}
 The first obvious consequence of Theorem \ref{th.1} is that the singular continuous spectrum
 of $H^D$ is then empty. 
 But it is actually a consequence of the general theorem in \cite{GER.MOU} revised in \cite{GER.MOU2}.
 
 The second consequence is that we can state the LAP outside $\calT_{sm}\cup\{0\}$
 in the same terms as those of G\'erard and Nier in the old version \cite{GER.MOU}
 of their work. See also \cite{MOU.ABS,PER.SPE}.
 Let us consider a compact interval $I\subset\R^*\sauf \calT_{sm}$, and fix
 $\phi\in \calC^\infty_c(\R^*\sauf\calT_{sm})$ such that $\phi=1$ on a neighborhood of $I$.
 We denote by $A_\phi^{sa}\subset A_\phi\!^*$ a self-adjoint extension of $A_\phi$. 
 We define the abstract Besov space $\calB_A$ by
$$
 \calB_A = \{ f\in \calH;\: \|f\|_{\calB_A} := \sum_{j=0}^\infty r_j^{1/2} \| 1_{r_{j-1}\le |A_\phi^{sa}|\le r_j}f\| < \infty\}.
$$
 Its dual space $\calB_A\!^*$ is the completion of $\calHD$ by the following norm
$$ \|u\|_{\calB_A\!^*} = \sup_{j\ge 0}\:  r_j^{1/2} \|1_{r_{j-1} \le | A_\phi^{sa} | < r_j}u\| . $$
 For $s > 1/2$, the following inclusion relations hold :
\begin{eqnarray*}
 D((1+|A_\phi^{sa}|)^{s}) \subset \calB_A \subset D((1+|A_\phi^{sa}|)^{1/2}) \subset \calHD  \\
  \subset D((1+|A_\phi^{sa}|)^{-1/2})  \subset \calB_A\!^* \subset D((1+|A_\phi^{sa}|)^{-s}).
\end{eqnarray*}
 We can claim
\begin{coro}
\label{coro.1}
 (LAP on $\R^*\sauf \calT_{sm}$ in abstract Besov spaces.) We have
$$
 \sup_{\lambda \in I, \, \mu >0} \| R(\lambda\pm i\mu)f\|_{\calB_A \!^*} \le C_I \|f\|_{\calB_A}
  \quad \forall f\in \calB_A .
$$
%
%
  Moreover letting $s > 1/2$ then the limits
$$ \lim_{\epsl\to \pm0} (1+|A_\phi^{sa}|)^{-s} R(\lambda +i\epsl)  (1+|A_\phi^{sa}|)^{-s} $$
 exist in $B(\calHD)$ and are bounded, with uniform convergence according to $\lambda\in I$.
 The mapping $\R\sauf\calT \ni \lambda \mapsto R(\lambda \pm i0)$ is norm continuous
 in $B(D((1+|A_\phi^{sa}|)^{s}, D((1+|A_\phi^{sa}|)^{-s})$ and weakly continuous in $B(\calB_A,\calB_A\!^*)$.
\end{coro}
 For further developments we establish also the LAP in terms of the usual Besov spaces
 described by Isozaki and alii \cite{AND.SPE} with the restriction to spectral values outside the thresholds.
 Thus, we consider the case $\supp\phi\subset\R\sauf\calT$ in Theorem \ref{th.1}.
 We set $N = (N_1,N_2 ,N_3)$, $N_j=i\partial/\partial x_j$ and the self-adjoint operators
$$ |N| = \sqrt{N^2} = \sqrt{-\Delta} , \quad N^2 = \sum_{j=1}^3 N_j^2 = -\Delta \quad {\rm on} \:  \T^3,$$
 where $\Delta$ denotes the Laplacian on $\T^3 = [-\pi,\pi]^3$ with periodic boundary condition.
 Let $D'(\T^3,\C)=(\calC^\infty(\T^3,\C))'$ be the space of distribution on $\T^3$ and consider
 $D'(\T^3,\C^6) \approx (D'(\T^3,\C))^6 \approx (D'(\T^3,\R))^{12} $.
 We introduce the normed spaces
$$ \calH^s = \{u \in D'(\T^3,\C^6),\: \|u\|_s < \infty \} ,  \quad \|u\|_s := \|(1 + N^2)^{s/2}u\| ,\quad s \in \R,$$
 so $\calH^s$ is the completion of $D(|N|^s)$, the domain of $|N|^s$, with respect to
 the norm $\|u\|_s$ and we have $\calHD = \calH^0 = L^2(\T^3,\C^6)$.
 For $s\ge 0$ and $u\in \calC^\infty(\T^3,\C^6)$ we have $\|(1+|A_\phi^{sa}|)^su\| \le C \|u\|_s$
 where $C$ does not depend on $u$. Thus, the following inclusion relations hold :
$$
 \calH^s \subset D((1+|A_\phi^{sa}|)^s) \subset \calHD  \subset  D((1+|A_\phi^{sa}|)^{-s})
  \subset \calH^{-s}  \quad \forall s\ge 0.
$$
 Using the sequence $(r_j)_{j\ge -1}$ where $r_{-1}=0$, $r_j=2^j$ for $j\ge 0$ we define
 the Besov space $\calB$ by
$$
 \calB := \{ f\in \calHD;\: \|f\|_{\calB} := \sum_{j=0}^\infty r_j^{1/2} \| 1_{r_{j-1}\le |N|\le r_j}f\| < \infty\}.
$$
 Its dual space $\calB^*$ is the completion of $\calH$ by the following norm
$$ \|u\|_{\calB^*} = \sup_{j\ge 0}\:  r_j^{1/2} \|1_{r_{j-1} \le |N| < r_j} u\| . $$
 For $s > 1/2$, the following inclusion relations hold :
$$
 \calH^s \subset \calB \subset \calH^{1/2} \subset \calHD  \subset \calH^{-1/2}
  \subset \calB^* \subset \calH^{-s}.
$$
 Moreover, Lemma 2.8 of \cite{INV.ISO1} says that there is a constant $C > 0$ such that
$$ \|f\|_{\calB_A} \le C\|f\|_{\calB} \quad  \forall f \in \calB ,$$
 that is, $\calB \subset \calB_A$, and so, $\calB_A\!^* \subset \calB^*$.
 Hence, Corollary \ref{coro.1} can be extended as
\begin{coro}
\label{coro.2}
 (LAP on $\R\sauf \calT$ in usual Besov spaces.)
 We have
$$ \sup_{\lambda\in I,\; \mu>0} \|R(\lambda \pm i\mu)\|_{B(\calB;\calB^*)} < \infty, $$
 for all compact set $I\subset \R\sauf\calT$. Moreover letting $s > 1/2$ then
 the limits
$$ R(\lambda \pm i0) := \lim_{\epsl\searrow 0}R(\lambda \pm i\epsl) \in B(\calH^s; \calH^{-s})  $$
 exist and are bounded, with uniform convergence according to $\lambda\in I$.
 The mapping $\R\sauf\calT \ni \lambda \mapsto R(\lambda \pm i0)$ is norm continuous
 in $B(\calH^s; \calH^{-s})$ and weakly continuous in $B(\calB;\calB^*)$.
\end{coro}
 Another consequence of Theorem \ref{th.1} is the following extension of the LAP
 to any nonzero spectral value, thanks to a slight adaptation of \cite[Theorem 3.3]{GEO.COM}.
\begin{coro}
\label{coro.3}
 (LAP on $\R^*$.)
 Let $I\subset \R^*$ be a compact interval. There exists a constant $C_I$ such that
$$
  |(u,R(z) u)_{\calHD}| \le C_I  \|u\|^2_{D(\bar A_\phi)} \quad  \forall u\in D(\bar A_\phi),
$$
 for all $z=\lambda+i\mu$ with $\lambda\in I$, $\mu\neq 0$ real. Moreover if $z_1=\lambda_1+i\mu_1$,
 $z_2=\lambda_2+i\mu_2$ are two such numbers, and if $\mu_1$ and $\mu_2$ have the same sign,
 then
$$
  |(u,(R(z_1)-R(z_2)) u)_{\calHD}| \le C_I  |z_1-z_2|^{1/2} \|u\|^2_{D(\bar A_\phi)}  \quad  \forall u\in D(\bar A_\phi).
$$
 In particular, if $u\in D(\bar A_\phi)$ then the limits 
$$
 \lim_{\epsl\searrow 0^+} (u,R(\lambda \pm i\epsl)u)_{\calHD} =: (u,R(\lambda \pm i0)u)_{\calHD}  
$$
 exist uniformly in $\lambda\in I$, and, for all $\lambda_1,\lambda_2\in I$,
 we have
$$
 (u,(R(\lambda_1 \pm i0)-R(\lambda_2 \pm i0))u)_{\calHD} 
  \le C_I  |\lambda_1-\lambda_2|^{1/2} \|u\|^2_{D(\bar A_\phi)} .
$$
\end{coro}
 An immediate consequence of Corollary \ref{coro.3} is
\begin{coro}
\label{coro.4}
 The point spectrum $\sigma_p(H^D)$ of $H^D$ is reduced to $\{0\}$.
\end{coro}
 Before giving the proof of Theorem \ref{th.1}, we state the results for some natural class
 of perturbed Hamiltonians $H_V^D=H^D+V$, as done in \cite{GER.MOU}.
 We will simply recall some well known results in the Mourre theory (see \cite{MOU.ABS,PER.SPE})
 and refer the reader to the book \cite{AMR.GRO} for a complete exposition of the Mourre method.
 In particular a sharper version of Corollary \ref{coro.5} is given in \cite[Prop. 7.5.6]{AMR.GRO}.
\begin{coro}
\label{coro.5}
 Let $I\subset\R^*\sauf \calT_{sm}$ be a compact interval, and fix
 $\phi\in \calC^\infty_c(\R^*\sauf\calT_{sm})$ such that $\phi=1$ on a neighborhood of $I$.
 Let $V$ be a symmetric operator on $\calHD$ 
 so that
\begin{enumerate}
\item $VR(i)$ and $R(i)[V,iA_\phi]R(i)$ are compact.
\item $V\in C^{1,1}(A_\phi)$, i.e.,
$$ \int \| R(i)(e^{itA_\phi}[V,iA_\phi] e^{-itA_\phi}-[V,iA_\phi])R(i) \| \frac{\rmd t}t < \infty. $$
\end{enumerate}
 Then, setting $H^D_V:=H^D+V$, the following results hold:
\begin{enumerate}
\item
 There exists a constant $\delta>0$ and a compact operator $K$ so that,
%
$$ \phi(H^D_V)[H^D_V,iA_\phi]\phi(H^D_V) \ge \delta \phi^2(H^D_V)+K. $$
%
 Consequently point spectrum $\sigma_p(H^D_V)$ is of finite multiplicity in
 $\R^*\sauf\calT_{sm}$ and has no accumulation point in $\R^*\sauf\calT_{sm}$.
\item For each $\lambda\in I\sauf\sigma_p(H^D_V)$ there exist $\epsl>0$ and $c>0$ so that,
\begin{equation}
\label{ineq.Mourre3}
 1_{[\lambda-\epsl,\lambda+\epsl]}(H^D_V)[H^D_V,iA_\phi]1_{[\lambda-\epsl,\lambda+\epsl]}
  \ge c 1_{[\lambda-\epsl,\lambda+\epsl]}(H^D_V).
\end{equation}
\item The LAP for $H^D_V$ holds on $I\sauf\sigma_p(H^D_V)$: the limits
$$
 \lim_{\epsl\to \pm 0} (1+|A_\phi|)^{-s} R(\lambda \pm i\epsl) (1+|A_\phi|)^{-s}
$$
 exist and are bounded for all $s>1/2$.\\
 Consequently the singular continuous spectrum of $H^D_V$ is empty.
\item If the operator $(1+|A_\phi|)^{-s}V(1+|A_\phi|)^{-s}$ is bounded for some $s>1/2$,
 then the wave operators
%
$$ s-\lim_{t\to\pm\infty} e^{itH^D_V} e^{itH^D} 1_I(H^D)=: \Om_I^\pm $$
%
 exist and are asymptotically complete:
%
$$ 1_I^c(H^D_V)\calHD = \Om_I^\pm \calHD. $$
%
\end{enumerate}
\end{coro}

\section{First Spectral Properties}
\label{sec.4}
\subsection{Spectrum of $h^D(y)$}
 We introduce the new parameters:
\begin{eqnarray}
\label{def.beta}
 \bfbeta &=& \bfepsl \times \bfmu = (\beta_1,\beta_2,\beta_3)\\
\nonumber
 \bfalpha &=& (\alpha_1,\alpha_2,\alpha_3),\quad
  \alpha_1 := (\epsl_2  \mu_3 + \epsl_3 \mu_2)/2 \quad \mbox{and, c.p.},\\
\label{def.gamma}
 \gamma_1 &=&  \epsl_2 \epsl_3 \mu_2\mu_3  \quad \mbox{and c.p. :}\\
\nonumber
 \gamma_2 &=&  \epsl_1 \epsl_3 \mu_1\mu_3, \\
\nonumber
 \gamma_3 &=&  \epsl_1 \epsl_2 \mu_1\mu_2.
\end{eqnarray}
 (The abreviation "c.p." means "circular permutation"  so we have the other values by circular permutation,
 ex., $\beta_1 = \epsl_2 \mu_3 - \epsl_3 \mu_2)$.

 Let us describe the spectrum of $h^D(y)$.
\begin{lemma}
\label{lem.valp}
 We have
%
$$ \det(h^D(y)-k) = \det(\bfepsl M(y)\bfmu M(y)+k^2) $$
 and the factorization
%

$$ \det(h^D(y)-k)  = k^2(\tau^+(z)-k^2)(\tau^-(z)-k^2), $$
 with
\begin{equation}
\label{v.taupm}
 \tau^\pm =  \Psi_0\pm\sqrt{K_0} ,
\end{equation}
 where
\begin{eqnarray}
\label{def.K0}
   K_0(z) &=&    \frac14( \beta_1^2 z_1^2  - 2\beta_1\beta_2 z_1z_2) + c.p. ,\\
\label{def.Psi0}
\quad \Psi_0(z) &=& \bfalpha\cdot z \: := \,  \alpha_1 z_1 + \alpha_2 z_2 + \alpha_3 z_3 .
\end{eqnarray}
\end{lemma}
 Proof in Appendix A.
 Since the characteristic polynomial  $\det(H^D(x)-\lambda)$ depends on $x$
 via the new variable $z =  (z_1,z_2,z_3) = \sin^2 x \in [0,1]^3$ we set
$$
  p(z;\lambda) =  \det(H^D(x)-\lambda) = \det(\bfepsl M(y)\bfmu M(y)+\lambda^2).
$$
\begin{remark}
 If $z\neq 0$ then
$$
   \tau^+(z) = \Psi_0(z)+\sqrt{K_0(z)} \ge  \tau^-(z) = \Psi_0(z)-\sqrt{K_0(z)} >0.
$$
 Moreover there exists $C>0$ such that
%
$$ \tau^-(z) \ge C|z| , \quad z\in [0,1]^3. $$
%
\end{remark}
%

\medskip
 Since $\bfbeta\cdot \bfepsl=0$ and $\epsl_i>0$ for all $i\in [[1,3]]$ then there exists
 $j\in [[1,3]]$ such that $\beta_j \beta_i\le 0$ and $\beta_k \beta_i\ge 0$ for $i,k\neq j$.
 If two of the $\beta_j$'s vanish then $\bfbeta$ vanishes.
 Moreover $\bfbeta$ is replaced by $-\bfbeta$ if $\bfepsl$ and $\bfmu$ are exchanged,
 which involves the same analysis.
 Hence if $\bfbeta\neq0$ then we can assume without any restriction:
\begin{center}
 {\bf (A0)}: $\beta_1\ge \beta_2 > 0 >  \beta_3$ or $\beta_1> \beta_2 = 0 >  \beta_3$.
\end{center}
 The functions $\Psi_0$ and $K_0$ are homogeneous polynomials.
 The relation $K_0\equiv 0$ is equivalent to $\bfbeta=0$ which is the special case
 where $\epsl$ and $\mu$ are proportional.
 If one of the $\beta_i$'s vanishes, so, under Asumption (A0), $\beta_2=0$, then
 $\sqrt{K_0(z)}$ is polynomial.
 Thus, the functions $\R^3\ni z \mapsto \tau^\pm(z)$ are homogeneous analytical complex functions
 with branch at $z=0$ and at points $z$ for which $K_0(z)=0$ if $\beta_2\neq 0$ for all $i$,
 and with branch point at $z=0$ only if $\beta_2\neq 0$.

 If $z=0$ then $h^D(y)=0_{6\times 6}$ and all the eigenvalues vanish.
 Let us consider the case $z\neq 0$.
\begin{theor} (Spectrum of $h^D(y)$.)
\label{th.spectreH}
 Let  $z\in [0,+\infty)^3\sauf \{0_{\R^3}\}$. Then $0$ is a double eigenvalue with eigenvectors
 $(y_1,y_2,y_3,0,0,0)\equiv y\otimes 0_{\C^3}$ and
 $(0,0,0,y_1,y_2,y_3)\equiv 0_{\C^3}\otimes y$.

 Assume $\bfbeta=0$. Then $K_0\equiv 0$ and all the eigenvalues have multiplicity two.
 Moreover, the nonzero eigenvalues of $h^D(y)$ are
$$
 \pm\sqrt{\tau^+(z)}=\pm\sqrt{\tau^-(z)} = \pm \sqrt{\epsl_2\mu_3 z_1
  + \epsl_3\mu_1 z_2+\epsl_1\mu_2 z_3}.
$$
 Assume $\bfbeta\neq0$ (so (A0) holds).
 Then the nonzero eigenvalues of $h^D(y)$ are
\begin{itemize}
\item $\pm\sqrt{\tau^+(z)}$, simple iff $K_0(z)\neq 0$,
\item $\pm\sqrt{\tau^-(z)}$, simple iff $K_0(z)\neq 0$.
\item $\pm\sqrt{\tau^+(z)}=\pm\sqrt{\tau^-(z)}$, double iff $K_0(z)= 0$.
\end{itemize}
 Assume $\bfbeta\neq0$ (so (A0) holds) with $\beta_2=0$. 
 Then, $\tau^+$ and $\tau^-$ are linear according to $z$:
\begin{eqnarray}
\label{val.tau+-anal}
 \tau^+(z) &=&  \epsl_2\mu_3 z_1+ \epsl_3\mu_1 z_2 + \epsl_2\mu_1 z_3,\\
\label{val.tau--anal}
 \tau^-(z) &=& \epsl_3\mu_2 z_1+ \epsl_3\mu_1 z_2 + \epsl_1\mu_2 z_3.
\end{eqnarray}
 (Hence we observe that:\\
 If $(z_1,z_3)\neq 0_{\R^2}$, then the nonzero eigenvalues of $h^D(y)$ are
 $\pm\sqrt{\tau^+(z)}$, simple.
 If $(z_1,z_3)= 0_{\R^2}$ and $z_2\neq 0$, then the nonzero eigenvalues of $h^D(y)$ are
$$
 \pm\sqrt{\tau^+(z)}=\pm\sqrt{\tau^-(z)}=\pm \sqrt{\alpha_2} |y_2| =\pm \sqrt{\epsl_3\mu_1} |y_2| ,
$$
 double.)
\end{theor}
 The Proof of Theorem \ref{th.spectreH} follows from\refq{def.K0} and\refq{def.Psi0}.
\begin{lemma}
\label{lem.K0=0}
 Let us assume $\bfbeta\neq0$ (so (A0) holds). We have
$$ K_0^{-1}(\{0\}):= \{z\in [0,1]^3 \, ; \: K_0(z)=0\} = \{t(\beta_2,\beta_1,0)\, ; \: 0 \le t\le \frac1{\beta_1}\} .$$
\end{lemma}
 The Proof is let to the reader.

 We set
$$ \textcolor{red}{\lambda_\pm = \max\{\sqrt{\tau^\pm(z)} }\,; \; z \in [0,1]^3\}\in (0,+\infty) . $$
 The following result is a consequence of Theorem \ref{th.spectreH} and of relation\refq{val1.sigma}:
\begin{prop}
\label{p.HD-spectre}
 1) The operator $H^D$ admits $0$ as eigenvalue of infinite order.\\
 2) The spectrum of $H^D$ is
$$ \sigma(H^D) = [-\textcolor{red}{\lambda_+},\textcolor{red}{\lambda_+}]. $$
\end{prop}
 Proof in Appendix A. 

\subsection{Stratification and thresholds  }
\label{sec.seuils}
 Following \cite{GER.MOU} the set of energy-momentum is
%
$$ \Sigma = \{(\lambda,x)\,| \; \lambda\in \sigma(H^D(x))\} \subset \sigma(H^D) \times \T^3. $$
%
 We have $(\lambda,x)\in\Sigma\egal p(z;\lambda)=0$. 
 We consider the canonical projections: 
\begin{eqnarray*}
 P_M &:& \R\times\T^3 \ni (\lambda,x) \mapsto x\in \T^3,\\
 P_\R &:& \R\times\T^3 \ni (\lambda,x) \mapsto \lambda\in \R.
\end{eqnarray*}
 It is clear that $P_\R|_{\Sigma}$ is a proper map.
 The spectrum $\sigma(H^D(x))$ of $H^D(x)$ is discrete and depends continuously on $x$.
 The operators $H^D(x)$ are the fibers and the space $\T^3$ is the momentum space.\\
 The set of energy-momentum $\Sigma$ admits the partition
$$ \Sigma = \cup_{i=1}^6 \Sigma_i, $$
 where $\Sigma_i$ is the semi-analytical set of elements $(\lambda,x)$ for which
 $\lambda$ is an eigenvalue of multiplicity $i$ of $H^D(x)$.
 We set
$$ \calX_j =  P_M(\Sigma_j),\quad j\ge 1. $$
 We see that 
 $\Sigma_j=\emptyset$ for $j=3,4,5$, $\Sigma_6= \{0\}\times \calX_0$ so
 $\calX_6=\calX_0$, $\calX_j=\emptyset$ for $j=3,4,5$.
 Moroever,
\begin{eqnarray*}
 \Sigma_1 &=& \Sigma_1^+ \cup \Sigma_1^-, \\
 \Sigma_1^\pm &:=& \{(\lambda,x);\: 0\neq \lambda^2=\tau^\pm(z)\neq \tau^\mp(z)\},\\
 \Sigma_2 &=& \{(\lambda,x); \: 0\neq \lambda^2=\tau^+(z)=\tau^-(z)\} .
\end{eqnarray*}
 If $\bfbeta=0$ then $\Sigma_1=\emptyset$ and $(\lambda,x)\in\Sigma_2$ iff
 $z\neq 0$ and $\lambda^2=\Psi_0(z)$.\\
 If $\bfbeta\neq0$ holds then $(\lambda,x)\in\Sigma_1$ iff $K_0(z)\neq 0$ and
 $\lambda^2\in \{\tau^+(z),\tau^-(z)\}$.
 
 Let us define the set of thresholds, $\calT$. For a more general definition of $\calT$
 (which may defer from ours), see \cite{GER.MOU}.
 We set
\begin{eqnarray*}
 \Sigma_1^* &:=&  \Sigma_1^{*+} \cup  \Sigma_1^{*-},\\
 \Sigma_1^{*\pm} &:=& \{(\lambda,x)\in \Sigma_1; \: \lambda^2=\tau^\pm(z), \nabla_x\tau^\pm(z)=0\},\\
 \Sigma_2^* &:=& \{(\lambda,x)\in \Sigma_2; \: \lambda^2=\Psi_0(z), 
  \mbox{ and $\nabla_x\Psi_0(z)$ is normal to $\calX_2$ at $x$}\},
\end{eqnarray*}
%
 then, 
\begin{eqnarray*}
 \calT_j &:=& P_\R(\Sigma_j^*), \\
 \calX_j^* &:=& P_M(\Sigma_j^*),\\
 \calZ^*_j &:=& \sin^2(\calX_j^*), \quad j=1,2.
\end{eqnarray*}
 Observing that $P_\R(\Sigma_6)=\{0\}$, we set the set of thresholds, $\calT$, as
$$ \calT :=  \{0\} \cup \calT_1 \cup  \calT_2. $$
 Let us describe $\calT$. Setting
$$ \calX^* = \calX_1^* \cup \calX_2^* , $$
 we see that $\lambda$ is a threshold iff there exists $z^*\in  \sin^2(\calX^*)$ such that $p(z^*;\lambda)=0$.
 Setting
\begin{eqnarray*}
 \calT_1^\pm &:=& P_\R(\Sigma_1^{*\pm}),\\
 \calX_1^{*\pm} &:=& P_M(\Sigma_1^{*\pm}),\\
 \calZ_1^{*\pm} &:=& \sin^2(\calX_1^{*\pm}),
\end{eqnarray*}
 we have
\begin{eqnarray*}
 \calT_1 &=& \calT_1^+ \cup \calT_1^- ,\\
 \calX_1^* &=& \calX^{*+}_1\cup \calX^{*-}_1, \\
 \sin^2(\calX^*_1) &=& \calZ_1^{*+}\cup \calZ_1^{*-},
\end{eqnarray*}
 so
\begin{equation}
\label{caract.T}
 \calT =  \calT_1^+ \cup  \calT_1^- \cup  \calT_2 \cup \{0\} .
\end{equation}
 Obviously $\calT$ is symmetric to 0 so we shall analyse $H^D$ at the positive thresholds only.
\begin{lemma}
\label{lem.dtau=0}
\begin{itemize}
 \item [A)] Assume $\bfbeta=0$.
 Then we have $\partial_{z_i} \tau^\pm(z) = \partial_{z_i} \Psi_0(z) > 0$ for all $i$.
 \item [B)] Assume $\bfbeta\neq0$ (so (A0) holds.) We set
\begin{eqnarray}
\label{def.nu}
 \nu &:=&  \frac{2\alpha_3\sqrt{\beta_1\beta_2} - \sqrt{\gamma_3}(\beta_1 + \beta_2)}
 {|\beta_3| \sqrt{\gamma_3}}.
\end{eqnarray}
 Let $z\in [0,1]^3$ such that $K_0(z)\neq 0$.
 \begin{itemize}
 \item[1)] We have $\partial_{z_i} \tau^+(z)> 0$ for $i=1,2,3$, and $\partial_{z_i} \tau^-(z)> 0$
  for $i=1,2$.
 \item[2)] 
  \begin{itemize}
   \item [a)] Assume $\beta_2=0$ (so $\nu<0$). Then $\partial_{z_3} \tau^-(z)> 0$.
   \item [b)] Assume $\beta_2>0$.
   \begin{itemize} 
    \item[(i)] If $z_1=0$ or $z_2=0$ then $\partial_{z_3} \tau^-(z)> 0$.\\
    \item[(ii)] The derivative $\partial_{z_3} \tau^-(1,1,z_3)$ vanishes iff $z_3=\nu\in [0,1]$,
  and if $z_3\neq \nu$ then $\partial_{z_3} \tau^-(1,1,z_3)$ has the same sign than $z_3-\nu$ .
   \end{itemize}
 \end{itemize}
\end{itemize}
\end{itemize}
\end{lemma}
 Proof in Appendix A.
\begin{remark}
\label{rem.nu}
 Let $\tilde\nu\in\R$ there exist $\bfepsl$ and $\bfmu$ such that $\nu=\tilde\nu$.
 Proof in Appendix A.
\end{remark}
 Lemma \ref{lem.dtau=0} implies that the thresholds of the analytically fibered
 family $(H^D(x),x\in\T^3)$ come from the values $x\in \T^3$ such that $\partial_{x_i} z_i(x)=0$
 at least for $i=1,2$, so $z_1,z_2\in \{0,1\}$, and, in addition, we have $z_3\in \{0,\nu,1\}$.

 We can determine now the set $\calT$ of thresholds.
 Setting
\begin{eqnarray*}
 Z_{\{0,1\}} &=& \{0,1\}^3, \quad Z_{\{0,1\}}^* = Z_{\{0,1\}} \sauf \{0_{\R^3}\}, \\
 X_{\{0,1\}} &=& \{ x\in \T^3; \; z\in Z_{\{0,1\}}\}, \quad X_{\{0,1\}}^* = \{ x\in \T^3; \; z\in Z_{\{0,1\}}^*\},
\end{eqnarray*}
 we obtain the following (remember also\refq{caract.T})
\begin{lemma}
\label{lem.seuils}
\begin{itemize}
 \item [1)] Case $\bfbeta=0$. 
 We have $K_0\equiv 0$, $\calX_1=\emptyset$, $\Sigma_1=\emptyset$, $\calX_2=\T^3_0$.
 Then $ \sin^2(\calX^*_2) = Z_{\{0,1\}}^*$, $\calT_1=\emptyset$, $ \calT= \{0\}\cup \calT_2$ and
$$ \calT_2 \cap \R^+ = \{\sqrt{\Psi_0(z)}; z\in Z_{\{0,1\}}^*\}  . $$
 \item [2)]  Case $\bfbeta\neq 0$ (so (A0) holds). 
 We set $\calZ^*_{\nu}=\{(1,1,\nu)\}\cap [0,1]^3$.
 Then the sets $\calX_j$, $j=1,2$, are not trivial (remember Lemma \ref{lem.K0=0}).
 We have $\calZ_1^{*+}= Z_{\{0,1\}}^* \sauf \{(\frac{\beta_2}{\beta_1},1,0)\}$,
 $\calZ_1^{*-} = Z_{\{0,1\}}^* \cup \calZ^*_{\nu} \sauf \{(\frac{\beta_2}{\beta_1},1,0)\}$
 and $\sin^2(\calX^*_2)= \{(\frac{\beta_2}{\beta_1},1,0)\}$, so
\begin{eqnarray*}
 \calT_1^+ \cap \R^+ &=& \{\sqrt{\tau^+(z)},\; z\in Z_{\{0,1\}}^* , \; z\neq (\frac{\beta_2}{\beta_1},1,0)\} ,\\
 \calT_1^- \cap \R^+ &=& \{\sqrt{\tau^-(z)},\; z\in Z_{\{0,1\}}^* \cup \calZ^*_{\nu}, \; z\neq (\frac{\beta_2}{\beta_1},1,0)\} ,\\
 \calT_2 \cap \R^+ &=& \{\sqrt{\Psi_0(\frac{\beta_2}{\beta_1},1,0)}\}.
\end{eqnarray*}
\end{itemize}
\end{lemma}
 Proof in Appendix A.

\begin{remark}
 We have $\calX^{*+}_1\subset \calX^{*-}_1$ and the inclusion is an equality
 if and only if $\nu\not\in (0,1)$.
\end{remark}
\begin{remark}
 Let us consider the case $\beta_2=0$.
 By definition of $\calT_2$ we have
%
$$ \calT_2 \cap \R^+ = \{\sqrt{\Psi_0(z)}; \; z\neq 0 , \; \nabla_x \sqrt{\Psi_0(z)}=0\}. $$
%
 Then, the eigenvalues $\sqrt{\tau^\pm}$ are analytic in $[0,+\infty)^3\sauf\{0_{\R^3}\}$,
 since the functions $\tau^\pm$ are linear and positive in $[0,+\infty)^3\sauf \{0_{\R^3}\}$.
 Hence it would be an alternative to replace $\calT_2$ by $\calT'_2$ where we set
\begin{eqnarray*}
 \calT'_2 &:=& \calT^{'+}_2 \cup \calT^{'-}_2, \quad \calT^{'-}_2 := - \calT^{'+}_2,\\
 \calT^{'+}_2 &:=& \{\sqrt{\tau^\pm(z)}; \; z\neq 0 , \; \nabla_x \sqrt{\tau^\pm(z)}=0\}=\calT'_2\cap \R^+.
\end{eqnarray*}
 But, thanks to Points B) 1) and 2) a) of Lemma \ref{lem.dtau=0} we obtain
$$ \calT^{'+}_2 =  \{\sqrt{\tau^\pm(z)}; \; z\in \{0,1\}^3, \; z\neq 0_{\R^3} \} = \calT_2\cap \R^+, $$
 so the sets $\calT'_2$ and $\calT_2$ coincide.
\end{remark}

\section{The conjugated operator}
\label{sec.5}
 In this section we consider a set $\calT'\subset \calT$ and a function
 $\phi\in\calC^\infty_c((0,+\infty)\sauf\calT';\R)$.
 We construct an adequate conjugated operator $A_\phi$ to $H$ on $\supp\phi$.

\subsection{Eigenprojectors}
 If $\bfbeta=0$ then $\Sigma_1=\emptyset$ and the function $\sqrt{\Psi_0(y^2)}$ is analytic
 in $\R^3\sauf\{0_{\R^3}\}$.
 The associated orthogonal eigenprojection
\begin{equation}
\label{def.pi2}
 \pi_2(y) := \frac1{2i\pi} \int_{\calC} (h^D(y)-\zeta)^{-1} d\zeta  \quad \forall y\neq 0, 
\end{equation}
 where $\calC\subset\C$ is a complex contour containing $\sqrt{\Psi_0(z)}$ but not $0$,
 is then analytic in $\R^3\sauf\{0_{\R^3}\}$ and has range two.
 
 Let us assume $\bfbeta\neq 0$.
 Let us denote by $\pi_1^\pm(y)$ the orthogonal eigenprojection on $\ker(h^D(y)-\sqrt{\tau^\pm(z)})$,
 i.e.,
$$ \pi_1^\pm(y) := \frac1{2i\pi} \int_{\calC} (h^D(y)-\zeta)^{-1} d\zeta \quad \forall y\in \sin(\calX_1) , $$
 where $\calC$ is a contour containing $\sqrt{\tau^\pm(z)}$ but no other eigenvalue of $h^D(y)$.
 Let again $\pi_2(y)$ be defined by\refq{def.pi2} where now $\calC$ is a contour containing both
 $\sqrt{\tau^+(z)}$ and $\sqrt{\tau^-(z)}$ but no other eigenvalue. 
 Thus $\pi_2(y)$ is the orthogonal eigenprojection  on $\ker(h^D(y)-\sqrt{\tau^+(z)}) + \ker(h^D(y)-\sqrt{\tau^-(z)})$,
 and 
%
$$ \pi_2(y)=\pi_1^+(y) \oplus \pi_1^-(y) \quad  \forall y\in \sin(\calX_1). $$
%
 Each $\pi_1^\pm(y)$ has range one and $\pi_2(y)$ has range two. 
 Each $\pi_1^\pm$ is analytic at $y\in \sin(\calX_1)$ (as subset of $\R^3$),
 $\pi_2$ is analytic on $\sin(\R^3)\sauf\{0_{\R^3}\}$ ($\subset\R^3$), and
\begin{eqnarray*}
 h^D(y)\pi_2(y) &=& \sqrt{\tau^+(z)}\pi_1^+(y)+\sqrt{\tau^-(z)}\pi_1^-(y) \quad \forall y\in \sin(\calX_1),\\
 h^D(y)\pi_2(y) &=& \sqrt{\Psi_0(z)}\pi_2(y) \quad \forall y\in \sin(\calX_2).
\end{eqnarray*}

\subsection{Global tangent field to $\calX_2$}
\label{sec.tangentfield}
 If $\beta_2=0$ then the $\tau^\pm$'s and $\pi_1^\pm$'s extend analytically into
 $\R^3\sauf\{0_{\R^3}\}$ with the relation
%
$$  \pi_2(y)=\pi_1^+(y) + \pi_1^-(y) \quad \forall y\neq 0_{\R^3}. $$
%
 In addition, in Case (A0) (with $\beta_2=0$), the sum $\pi_1^+(y) + \pi_1^-(y)$ is direct.

 Let us assume $\beta_2\neq 0$ (with Assumption (A0)) and make a precise description
 of $\calX_2$.
 A point $x\in\T^3$ belongs to $\calX_2$ iff $z\neq 0$ and $z_3=0=\beta_1 z_1 - \beta_2 z_2$.
 The last relation can be written
%
$$ \beta_1 y_1^2 = \beta_2 y_2^2, \quad y_2\in [-1,1]\sauf\{0\}. $$
%
 When an nonzero eigenvalue of $H^D(x)$ (respect., of $h^D(y)$) is not simple then
 the stratification method explained in \cite{GER.MOU} involves a tangential vector field
 to the set $\calX_2$ (respect., to $\sin(\calX_2)$):
%
 $w(x) := (\sin(x_1)\cos(x_2),\cos(x_1)\sin(x_2),0)$,
%
 (respect., $\tilde w(y) := (y_1,y_2,0)$ ).
%
%
 We observe that $|w(x)|\neq 0$ for all $x\in \calX_2\sauf\calX^*$,
 and $|w(x)|\neq 0$ for all $x\in \calX_2$ if $\beta_2\in (0,\beta_1)$.
 If $\beta_2=\beta_1$ then $w$ vanishes at all $x^*\in \calX_2^*$ since $z^*_1=z^*_2=1$.

\medskip
 We introduce the following notations.
 Letting a function $f$: $\T^3$ or $\R^3$ $\to \C^n$ and a vector field $v(x)=(v_j(x))_{1\le j\le n}\in\C^n$,
 then $v\cdot\nabla_x f$ is the vectorial function
 $x\mapsto \sum_{j=1}^n v_j(x) \partial_{x_j} f(x) \in \C^n$.
 We set also
\begin{eqnarray*}
  f_w &:=& w\cdot\nabla_x f,\\
\tilde f_{\tilde w} &:=& \tilde w\cdot\nabla_y f.
\end{eqnarray*}
 We thus have
\begin{equation}
\label{val.fw}
 f_w(x) = \cos(x_1)\cos(x_2) \tilde f_{\tilde w}(y) .
\end{equation}

\subsection{First cut--off functions}
\label{subsec.cutoff}
 We consider the following metric on $\T^3\approx (\R/(2\pi\Z))^3$:
$$ \rmd_0(x,x^*)=|e^{ix}-e^{ix^*}|  \quad x^*, \, x\in \T^3 .$$
 We denote $\rmd_0(x,E)=\inf\{\rmd_0(x,x^*)$ $|\, x^*\in E\}$ when $E\subset\T^3$. 
 We consider a  cut--off function $\varphi_1\in \calC^\infty(\R;[0,1])$ such that
 $\supp \varphi_1\subset \{s\in\R;\: |s|<1\}$ and $\varphi_1=1$ in $\{s;\: |s|<1/2\}$.
 Let $b,b_0$ with $0<b<b_0/2$ two small parameters which will be precised later.
 We separate the eigenvalue 0 from $\R^+$, and, equivalently, $\calX_0$ from $\T^3$,
 with the cut--off function $\chi_0(x) := \varphi_1(|z|/b_0)$.
 Since $\supp\phi\subset (0,\infty)$, we can fix $b_0$ sufficiently small such that:
%
$$ \{\sqrt{\tau^+(z)},\sqrt{\tau^-(z)}\}\cap \supp\phi  \neq\emptyset \Rightarrow  \chi_0(x)=1. $$
%
 In addition we set
\begin{eqnarray*}
 \chi_{x^*}(x) &:=& \varphi_1(\rmd_0(x,x^*)/b) \quad x^*, \, x\in \T^3,\\
\chi^{*+}(x) &:=& (1-\chi_0(x)) \Pi_{x^*\in \calX^{*+}_1} (1-\chi_{x^*}) , \\
\chi^{*-}(x) &:=& (1-\chi_0(x)) \Pi_{x^*\in \calX^{*-}_1} (1-\chi_{x^*}),\\
 \chi^*(x)&:=& (1-\chi_0(x)) \Pi_{x^*\in \calX^*} (1-\chi_{x^*}),
\end{eqnarray*}
 so $\chi^*$ vanishes in $\{x\in \T^3; \rmd_0(x,\calX^*)<b/2\}$ and in $\{x\in \T^3; |z|<b_0/2\}$;
 we have also $\chi^*(x)=1$ if $\rmd_0(x,\calX^*)>b$ and $|z|>b_0$.
 It means that $\chi^*$ is a smooth cut-off function localizing in the complement
 of $\calX^*\cup \calX_0$, and, since $\calX^*\cup \calX_0$ is a discrete set (and finite),
 we then have, for $b_0>0$ sufficiently small,
\begin{eqnarray*}
 1-\chi^*(x) =  \chi_0(x) + \sum_{x^*\in \calX^*} \chi_{x^*}(x),\\
 1-\chi^{*\pm}(x) =  \chi_0(x) + \sum_{x^*\in \calX^{*\pm}_1} \chi_{x^*}(x).
\end{eqnarray*}
\subsection{The conjugated operator outside thresholds}
\label{ssec.out}
 Case $\bfbeta=0$. 
 Remember that we have $\pi_1^+=\pi_1^-=\pi_2$ which is analytic in $\R^3\sauf \{0_{\R^3}\}$.
 We set, for $u\in \calC^\infty(\T^3)$,  $x\in \T^3$,
%
$$
  \calA_{out} u\: (x) :=  
    i \chi^*(x) \pi_2(y) \frac{\nabla_x \sqrt{\Psi_0(z)}}{|\nabla_x \sqrt{\Psi_0(z)}|^2}
     \nabla_x (\chi^*(x) \pi_2(y) u(x)).
$$
%
 Case $\bfbeta\neq0$ and $\beta_2=0$ (with Assumption (A0)).
 Remember that the functions $\tau^\pm(\cdot)$ are analytic in $\R^3$
 (see\refq{val.tau+-anal} and\refq{val.tau--anal}) and are positive in $[0,+\infty)^3\sauf \{0_{\R^3}\}$.
 Thus, the eigenvalues $\sqrt{\tau^\pm(\cdot)}$ are analytic in $[0,+\infty)^3\sauf \{0_{\R^3}\}$.
 We set, for $u\in \calC^\infty(\T^3)$,  $x\in \T^3$,
%
$$
   \calA_{out} u\: (x) := 
   \sum_\pm i \chi^*(x)  \pi_1^\pm(y) \frac{\nabla_x \sqrt{\tau^\pm(z)}}
  {|\nabla_x \sqrt{\tau^\pm(z)}|^2} \nabla_x (\chi^*(x) \pi_1^\pm(y) u(x)) .
$$
%
 Case $\beta_2\neq0$ (with Assumption (A0)).
 Firstly, we have $\calX_1^{*+} \subset \calX_1^{*-}$, but not necessarily the converse inclusion.
 Actually, if $\nu\in (0,1)$ (remember Lemma \ref{lem.seuils}) then
 the value $\sqrt{\tau^-(1,1,\nu)}$ is a threshold be not necessarily $\sqrt{\tau^+(1,1,\nu)}$,
 so we may have $x_\nu^* \in  \calX_1^{*-} \sauf \calX_1^{*+}$.\\
 Secondly, in aim to have $H^D\in \calC^\infty(A_{out})$, we need to separate $\calX_2\subset\partial \calX_1$
 from $\calX_1$, as explained in \cite{GER.MOU2}. 
 Since $\ove{\calX_2}=\calX_2\cup\calX_0=\{x\in \T^3; K_0(z)=0\}$ is compact then
 there exist two smooth cut-off functions, $\chi_1$ and $\chi_2$ in $\calC^\infty(\T^3;[0,1])$,
 such that $\supp \chi_2\subset \{x$; $\rmd_0(x,\calX_2)\le 2b\}$,
 $\chi_2(x)=1$ if $\rmd_0(x,\calX_2)\le b$, 
 $\chi_1(x)=1$ if $\rmd_0(x,\calX_2)\ge 3b$, and
 $\supp \chi_1\subset \{x$; $\rmd_0(x,\calX_2)\ge 2b\}$.
 Thus $\supp \chi_1\cap \supp \chi_2=\emptyset$ and $\chi_2=1$ on $\ove{\calX_2}$.
 We then set $\chi_3:= 1-\chi_1-\chi_2$ so $\supp\chi_3\subset \{x$; $b\le\rmd_0(x,\calX_2)\le 3b\}$,
 and
%
$$ \sum_{j=1}^3 \chi_j^2(x) >0 \quad \forall x\in \T^3 . $$
%
 See Figure \ref{Zones}.
\begin{figure}[h]
\caption{Cut-off}
\label{Zones}
\includegraphics[scale=0.4]{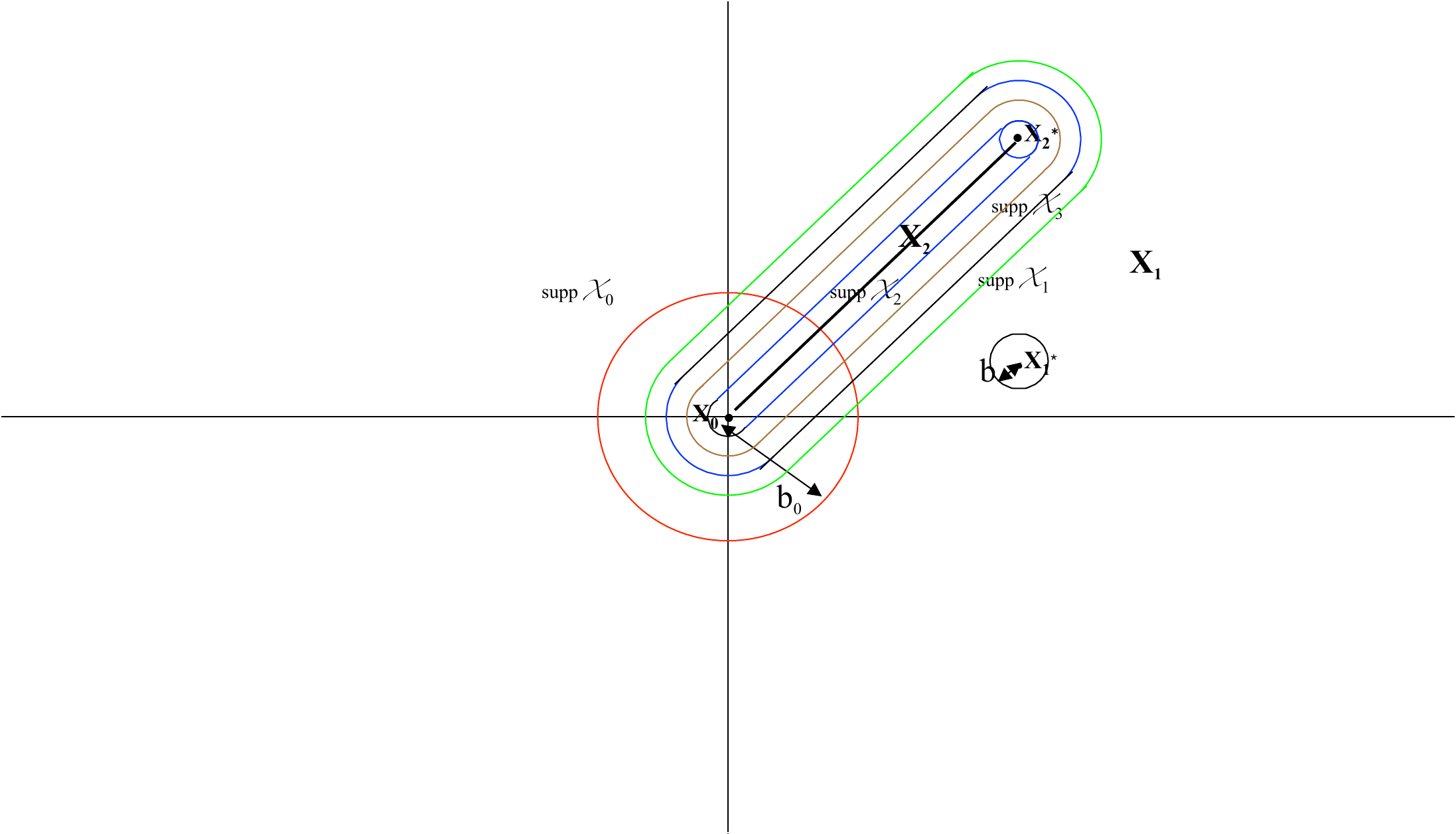} 
\end{figure}
 We have in addition ($b$ being sufficiently small)
\begin{eqnarray}
 \supp \chi_{x^*}\subset\supp \chi_1\sauf \supp \chi_3 \quad \forall x^*\in\calX_1^* ,\\
\label{prop.suppchi22}
 \supp \chi_{x^*}\subset\supp \chi_2\sauf \supp \chi_3 \quad \forall x^*\in\calX_2^* .
\label{prop.suppchi23}
\end{eqnarray}
\begin{eqnarray*}
 \chi^*_j(x) &:=& \chi^*(x)\chi_j(x) \quad j=2,3\\
 \chi^{*\pm}_1(x) &:=& \chi^{*\pm}(x)\chi_1(x).
\end{eqnarray*}
 (In fact, Relations \refq{prop.suppchi22} and \refq{prop.suppchi23} imply $\chi^*_3 = (1-\chi_0)\chi_3$.)
 The function $\chi^*_2$ is a smooth cut-off localizing in $\calX_2\sauf \calX^*$
 while $\chi^{*\pm}_1$ is a smooth cut-off localizing in $\calX_1\sauf \calX^{*\pm}_1$.

 For $u\in \calC^\infty(\T^3)$, $x\in \T^3$, we set,
\begin{eqnarray*}
  \calA_{out} u\: (x) &:=& 
  \sum_\pm  i \chi^{*\pm}_1 (x) \pi_1^\pm(y) \frac{\nabla_x \sqrt{\tau^\pm(z)}}
  {|\nabla_x \sqrt{\tau^\pm(z)}|^2} \nabla_x (\chi^{*\pm}_1(x) \pi_1^\pm(y) u(x)) \\
 && + i \chi^*_2(x)  (\sqrt{\Psi_0(z)}_w)^{-1} \pi_2(y) \; (\chi^*_2(x) \pi_2(y) u(x))_w\\
 && +  \sum_\pm  i \chi^*_3(x)  (\sqrt{\Psi_0(z)}_w)^{-1} \pi_1^\pm(y) \; (\chi^*_3(x) \pi_1^\pm(y) u(x))_w.
\end{eqnarray*}
\begin{remark}
 The function $\sqrt{\Psi_0(z)}_w$ may vanish at points of $\calX^*$ but not of $\calX_2\sauf\calX^*$
 so $\chi^*_2(x) (\sqrt{\Psi_0(z)}_w)$ is well-defined (for $b$ sufficiently small).
 For more details, see the proof of\refq{val.Tw} below.
 Similarly, the function $x\mapsto |\nabla_x \sqrt{\tau^\pm(z)}|$ is positive
 in $\supp\chi^{*\pm}_1$.
\end{remark}
 In each case we symmetrize $\calA_{out}$ by setting
%
$$  A_{out} := \calA_{out} + \calA_{out}^*, $$
%
 with domain $\calC^\infty(\T^3)$. Here $\calA_{out}^*$ is the hermitian conjugate of $\calA_{out}$.
 By observing that the mappings $x\mapsto \chi_j(x)\pi_1^\pm(y)$ for $j=1,3$, and
 $x\mapsto (1-\chi_0(x))\pi_2(y)$ are smooth, then
 $A_{out}$ is a symmetric first order differential operator in $x$ whose coefficients belong to
 $\calC^\infty(\T^3;\calL(\C^6)$). It is then essentially self-adjoint on $\calHD$ (see \cite[Lemma 3.10]{GER.MOU}).
 Since $D(H^D)=\calHD$, some possible problematic points of the Mourre Theory
 then become trivial (see \cite{GEO.COM}).
 
\subsection{"Punctual" Mourre's estimate outside thresholds} 
\label{sec.Hout}
 We set
$$ H_{1,out}(x) := [H^D,iA_{out}](x) .$$
 Similarly to proof of the Mourre's estimate in \cite{GER.MOU} we show that if the positive
 parameter $b$ is sufficiently small then $A_{out}$ is strictly conjugated to $H$ on $I$.
 
 Case $\beta_2\neq0$ under assumption (A0).
 Let $u\in \calC^\infty(\T^3)$, we have
\begin{eqnarray*}
  && -i\calA_{out}\circ H^D u\: (x) =\\
  &\quad&  \sum_\pm \chi^{*\pm}_1 (x) \pi_1^\pm(y) \frac{\nabla_x \sqrt{\tau^\pm(z)}}
  {|\nabla_x \sqrt{\tau^\pm(z)}|^2} \nabla_x (\chi^{*\pm}_1(x) \pi_1^\pm(y) h^D(y) u(x)) \\
 && + \chi^*_2(x) \pi_2(y) (\sqrt{\Psi_0(z)}_w)^{-1}\;  (\chi^*_2(x) \pi_2(y) h^D(y) u(x))_w \\
  && + \sum_\pm \chi^*_3(x) \pi_1^\pm(y) (\sqrt{\Psi_0(z)}_w)^{-1} \; (\chi^*_3(x) \pi_1^\pm(y)h^D(y) )_w.
\end{eqnarray*}
 By using 
%
$$
 (\pi_1^\pm(y))^2=\pi_1^\pm(y), \quad \pi_1^\pm(y) h^D(y) = h^D(y)\pi_1^\pm(y),
\quad \pi_2(y) h^D(y)=  h^D(y)\pi_2(y), 
$$
%
 we obtain the expression of $[H^D,i\calA_{out}]$ as a multiplication operator:
\begin{eqnarray*}
  [H^D,i\calA_{out}](x) &=&  (iH^D \circ \calA_{out}-i\calA_{out}\circ H^D)(x) \\
 &=&  \sum_\pm (\chi^{*\pm}_1(x))^2 \frac{\nabla_x \sqrt{\tau^\pm(z)}}
  {|\nabla_x \sqrt{\tau^\pm(z)}|^2} \pi_1^\pm(y) \nabla_x (\pi_1^\pm(y) h^D(y)) \pi_1^\pm(y) \\
 && + (\chi^*_2(x))^2 (\sqrt{\Psi_0(z)}_w)^{-1} \; \pi_2(y)   (h^D(y) \pi_2(y))_w \pi_2(y) \\ 
 &&  + \sum_\pm (\chi^*_3(x))^2  (\sqrt{\Psi_0(z)}_w)^{-1} \; \pi_1^\pm(y)(h^D(y) \pi_1^\pm(y))_w  \pi_1^\pm(y).
\end{eqnarray*}
 In $\calX_1$ we have
%
$$
 \pi_1^\pm(y) \nabla_x (h^D(y) \pi_1^\pm(y)) \pi_1^\pm(y) =
  \nabla_x \sqrt{\tau^\pm(z)} \pi_1^\pm(y),
$$
%
 so,
%
$$
 |\nabla_x \sqrt{\tau^\pm}(z)|^{-2} \pi_1^\pm(y)\nabla_x \sqrt{\tau^\pm}(z)
  \cdot \nabla_x (h^D(y) \pi_1^\pm(y))  \pi_1^\pm(y) =  \pi_1^\pm(y),
$$
%
%
$$
 \sum_\pm \pi_1^\pm(y) (\sqrt{\tau^\pm}_w(x))^{-1} w \cdot \nabla_x (h^D(y) \pi_1^\pm(y))
 = \sum_\pm \pi_1^\pm(y) = \pi_2(y).
$$
 Let us make the following computations near $\calX_2$, precisely, in $\supp\chi_2^*\cup\supp\chi^*_3$.
 Setting $\xi(y) := h^D(y)\pi_2(y)-\sqrt{\Psi_0(z)}\pi_2(y)$, we have
$$
  \pi_2(y)\; w(x) \cdot \nabla_x (h^D(y) \pi_2(y)) \pi_2(y)  =  \sqrt{\Psi_0(z)}(x) \pi_2(y) +  \pi_2(y)\; \xi_w(x)  \pi_2(y).
$$
 Thus,
\begin{eqnarray}
\nonumber
 \frac12 H_{1,out}(x) & = & \sum_\pm (\chi^{*\pm}_1(x))^2 \pi_1^\pm(y)  + ((\chi^*_2(x))^2 + (\chi^*_3(x))^2) \pi_2(y)  \\
 &&  + (\chi^*_2(x))^2  \pi_2(y) (\sqrt{\Psi_0(z)})^{-1}\; \xi_w(x)\pi_2(y).
\label{valfin3.HDAb}
\end{eqnarray}
 For $x\in\calX_2$ we have $\xi(y)=0$, so, since $\tilde w$ is a tangent field
 to $\sin(\calX_2)$,
$$ \tilde \xi_{\tilde w}(y) = 0 , \quad \forall x\in\calX_2.$$
 For $x\in\calX_2$ we have
\begin{equation}
\label{val.Tw}
 \tilde{\sqrt{\Psi_0(z)}}_{\tilde w} = (\Psi_0(z))^{-1/2}(\alpha_1 z_1 + \alpha_2 z_2)>0.
\end{equation}
 In addition, since the relation
%
$$ (\sqrt{\Psi_0(z)}_w)^{-1} \xi_w(x) = (\tilde{\sqrt{\Psi_0(z)}}_{\tilde w})^{-1} \tilde \xi_{\tilde w}(y) $$
%
 holds true for $x\not\in\calX^*$ and $\tilde{\sqrt{\Psi_0(z)}}_{\tilde w} \neq 0$,
 then the function $ (\sqrt{\Psi_0(z)}_w)^{-1} \xi_w$ is defined and is smooth in the compact set
 $\supp(1-\chi_0)\chi_2 \subset \supp\chi_2 \subset \{x\in \T^3$; $\rmd_0(x,\calX_2)\le 2b\}$,
 and vanishes on $\calX_2$.
 Hence, for $b$ sufficiently small, we have
\begin{equation}
\label{comp.Twxiw}
 \| (\sqrt{\Psi_0(z)}_w)^{-1} \xi_w(x)\|_\infty< \frac12 \quad \forall x\in \supp(1-\chi_0)\chi_2, 
\end{equation}
 where $\|\cdot\|_\infty$ denotes here the usual infinite norm on matrices.
 From\refq{valfin3.HDAb},\refq{comp.Twxiw}, we then obtain
\begin{equation}
 H_{1,out}(x) \ge  \sum_\pm (\chi^{*\pm}_1(x))^2 \pi_1^\pm(y) +( (\chi^*_2(x))^2  + (\chi^*_3(x))^2 )\pi_2(y).
\label{min3.H1out}
\end{equation}

\begin{remark}
 In the two other cases where $\beta_2=0$ we obtain
\begin{equation}
 \frac12 H_{1,out}(x)  = (\chi^*(x))^2 \pi_2(y),
\label{valfin12.HDAb}
\end{equation}
 so the punctual Mourre's estimate becomes simply
%
$$ H_{1,out}(x) \ge  2(\chi^*(x))^2 \pi_2(y). $$
%
\end{remark}

\subsection{Smoothness}
\label{ssc.smooth1}
 Relations\refq{valfin12.HDAb} and\refq{valfin3.HDAb} show that the symmetric form
 $H_{1,out}$ defined on $\calC^\infty(\T^3)$ is a multiplication operator on $\calHD$
 by smooth coefficients, so is bounded and closeable. 
 Thus, $[H_{1,out},iA_{out}]$ is a differential operator of order one at most.
 But when computing it's first order term we have to check only that $H_{1,out}$
 is commuting with each coefficient of the first order terms of $-i\calA_{out}(x)$.
 In fact, the possible problematic bracket arising from the calculation of
 $[H_{1,out},i\calA_{out}]$ is, in the case $\beta_2\neq0$, 
\begin{eqnarray*}
 {[}(\chi^*_2(x))^2 (\sqrt{\Psi_0(z)}_w)^{-1} \pi_2(y)\; \xi_w(x)\pi_2(y) &, 
 (\chi^{*\pm}(x))^2 \frac{\nabla_x \sqrt{\tau^\pm(z)}} 
  {|\nabla_x \sqrt{\tau^\pm(z)}|^2} \cdot\\
 & \cdot\pi_1^\pm(y) \nabla_x (h^D(y)\pi_1^\pm(y) ) \pi_1^\pm(y)].
\end{eqnarray*}
 But since $\chi_1\chi_2=0$ then this bracket vanishes.
 Hence, $[H_{1,out},iA_{out}]$ is a multiplication operator, is bounded in $\calH$,
 and we have $H^D\in \calC^2(A_{out})$. 
 By induction we see that $H^D\in \calC^\infty(A_{out})$. (See also \cite{GER.MOU2}.)

\subsection{The conjugated operator near thresholds}
\subsubsection{Enumeration of the different cases}
 Since our proof of the LAP at each threshold related to some $x^*\in \calX^*$ requires
 a special treatment which depends on the values of $\bfbeta$ and of $x^*$, we enumerate
 the different cases as follows.
\begin{itemize}
\item[Case 1] $\bfbeta=0$ and $x^*\in \calX_2^*$.
\item[Case 2]  $\bfbeta\neq 0$ (so (A0) holds) and $x^*\in \calX_1^*$.
\begin{itemize}
\item[Subcase 2-1] $\beta_2=0$.
\item[Subcase 2-2]  $\beta_2>0$.
\begin{itemize}
\item[Subcase 2-2-a] $x^*\in \calX_1^{*-}$ and $z^*_1=z^*_2=1$ and $z^*_3= \nu\in (0,1)$.
\item[Subcase 2-2-b]  $x^*\in \calX_1^{*-}$ and $z^*_1=z^*_2=1$ and $z^*_3\neq \nu$.
\item[Subcase 2-2-c]  $x^*\in \calX_1^{*-}$ and $z^*_1=z^*_2=1$ and $z^*_3= \nu\in\{0,1\}$.
\item[Subcase 2-2-d]  $x^*\in \calX_1^{*-}$ and ($z^*_1=0$ or $z^*_2=0$).
\item[Subcase 2-2-e]  $x^*\in \calX_1^{*+}$.
\end{itemize}
\end{itemize}
\item[Case 3]  $\bfbeta\neq 0$ and $x^*\in \calX_2^*$.
\begin{itemize}
\item[Subcase 3-1)]  $\beta_2=0$.
\item[SubCase 3-2]  $\beta_2=\beta_1$.
\item[SubCase 3-3]  $\beta_2\in (0,\beta_1)$.
\end{itemize}
\end{itemize}
\begin{remark}
\label{rem.2-2-c}
 In Case 2-2-c, if $\beta_1=\beta_2$ then $(1,1,0) \in \sin^2 \calX_2^*$ so $\nu=1$.
\end{remark}
\subsubsection{Behaviour of the eigenvalues of $H^D(x)$ at a threshold}
\label{sec.behave}
 We set $s^*_j=1-2z^*_j$ if $z^*_j\in \{0,1\}$, $j\in [[1,3]]$, so $s^*_j \in \{-1,1\}$.
 We set also $s_j=s^*_j$ for $j=1,2$.\\
 In Case 1 we set $V= \sqrt{}\circ\Psi_0\circ\sin^2$ and $s_3 :=s^*_3$.\\
 In Case 2 with $x^*\in \calX_1^{*\pm}$ and in Case 3-1 we set $V= \sqrt{}\circ\tau^\pm\circ\sin^2$ and\\
 - in Cases 2-1 and 2-2-d and 2-2-e, and 3-1 we set $s_3 :=s^*_3$;\\
 - in Case 2-2-a we set $s_3 :=1$;\\
 - in Case 2-2-b we set $s_3 := \sgn(z^*_3-\nu)s^*_3$;\\
\textcolor{red}{ - in Case 2-2-c we set $s_3 :=0$.}

\begin{lemma}
\label{lem.dV}
 In Cases 1 and 2-1 and 2-2-a and 2-2-b and 2-2-d and 2-2-e and 3-1) we have
\begin{equation}
\label{prop1.dV}
 \rmd V(x) = (\sum_{j=1}^3 C_j s_j (x_j-x^*_j) \rmd x_j)(1+O(\rmd_0(x,x^*))),
\end{equation}
 as $x\to x^*$, where $C_j>0$, $j=1,2,3$.\\
 In Case 2-2-c we have
\begin{eqnarray}
\nonumber
 \rmd V(x) &=& (-\sum_{j=1}^2 C_j (x_j-x^*_j) \rmd x_j  + C_3 (x_3-x^*_3)^3 \rmd x_3)
  (1+ O(\rmd_0(x,x^*))),\\
\label{prop2.dV}
\end{eqnarray}
 as $x\to x^*$, where $C_j>0$, $j=1,2,3$.
\end{lemma}
 Proof in Appendix B.
\begin{lemma}
\label{lem.wdH}
 Consider Cases 3-2 or 3-3 (i.e.,  Assumption (A0) with $\beta_2\neq0$ and $x^*\in \calX_2^*$).
 We then have the following estimates.\\
 In Case 3-2,
\begin{equation}
\label{est1.xiw3_2}
  (\sqrt{\Psi_0(z)}_w) = C(x_1-x^*_1)(x_2-x^*_2) (1+ O(\rmd_0(x,x^*))),
\end{equation}
 and, in Case 3-3,
\begin{equation}
\label{est1.xiw3_3}
  (\sqrt{\Psi_0(z)}_w) = C(x_2-x^*_2)(1+ O(\rmd_0(x,x^*))),
\end{equation} 
 for some $C\neq 0$ as $x\to x^*$.
\end{lemma}
 Proof in Appendix B.

\subsubsection{Partition of the set of thresholds}
\label{sec.taub}
 With the notations (the $s_j$'s notably) of Part \ref{sec.behave} we set
\textcolor{red}{ 
\begin{eqnarray*}
 \calT_{sa} &:=& \calT \sauf (\calT_{sm}\cup\{0\}),\\
 \calT_{sm} &:=& \calT_{sm}^+ \cup  -\calT_{sm}^+ ,
\end{eqnarray*}
 where, in Case $\bfbeta=0$,
\begin{eqnarray*}
 \calT_{sm}^+ &:=&  \{ \sqrt{\Psi_0(z^*)}; \; x^*\in \calX_2^* (=X_{\{0,1\}}^*) , \; s_1=s_2=s_3=\pm 1\},
\end{eqnarray*}
 and, in Case $\bfbeta=0$,
\begin{eqnarray*}
  \calT_{sm}^+ &:=&  \{ \sqrt{\tau^+(z^*)}, \sqrt{\tau^-(z^*)}; \; x^*\in \calX_1^*, \; s_1=s_2=s_3=\pm 1\}.
\end{eqnarray*}
 In fact we have,\\
 in Case $\bfbeta=0$:
\begin{eqnarray*}  
  \calT_{sm}^+ &=& \{\lambda_+=\lambda_-\},
\end{eqnarray*}
 in Case $\bfbeta\neq0$:
\begin{eqnarray*}
 \calT_{sm}^+  &=& \Big\{ \sqrt{\tau^+(1,1,1)}=\lambda_+, \sqrt{\tau^-(1,1,1)}=\lambda_-,
  \sqrt{\Psi_0(\frac{\beta_2}{\beta_1},1,0)}\Big\} \quad {\rm if } \quad \nu \le 0,\\
 \calT_{sm}^+ &=&  \Big\{ \sqrt{\tau^+(1,1,1)}=\lambda_+, \sqrt{\tau^-(1,1,1)}, \sqrt{\tau^-(1,1,0)},
 \sqrt{\Psi_0(\frac{\beta_2}{\beta_1},1,0)} \Big\},  \quad {\rm if } \quad \nu \in (0,1),\\
 \calT_{sm}^+  &=& \Big\{ \sqrt{\tau^+(1,1,1)}=\lambda_+,\sqrt{\tau^-(1,1,0)}=\lambda_-,
  \sqrt{\Psi_0(\frac{\beta_2}{\beta_1},1,0)}\Big\}  \quad  {\rm if } \quad \nu \ge 1 .
\end{eqnarray*}
\begin{remark}
 1) If $\bfbeta\neq 0$ then $\lambda_-<\lambda_+$. Actually, thanks to Lemma \ref{lem.dtau=0},
 if $\sqrt{\tau^-(z)}=\lambda_-$ then $z_1=z_2=1$ and $z_3\in \{0,1\}$. If $z_3=1$ then $K_0(z)\neq 0$ so 
 $\lambda_+=\sqrt{\tau^+(z)}>\sqrt{\tau^-(z)}=\lambda_-$ and if $z_3=0$ then
 $\lambda_+=\sqrt{\tau^+(1,1,1)}> \sqrt{\tau^+(1,1,0)}\ge \sqrt{\tau^-(z)}=\lambda_-$.\\
 2) If $\bfbeta\neq 0$ and $\nu\in (0,1)$ then $\lambda_-= \max \{\sqrt{\tau^-(1,1,1)}, \sqrt{\tau^-(1,1,0)}\}$,
 and each value $\sqrt{\tau^-(1,1,1)}$, $\sqrt{\tau^-(1,1,0)}$ is a local maximum of $\sqrt{}\circ \tau^- \circ \sin^2$.
\end{remark}
}

\subsubsection{New coordinate near an element of $\calX^*$}
\label{sec.p1}
 We give an approximation of a vector proportional to $\nabla _x V(x)$ (where $V$ is defined
 in Part \ref{sec.behave}) of the form $\nabla_x p_1$ near a point $x^*\in \calX_j^*$.
 We then give an approximation of a vector proportional to $w(x)$ near a point $x^*\in \calX_2^*$
 in Cases 3-2 and 3-3.

 With the notations of Lemma \ref{lem.dV}, in Cases 1, 2-1, 2-2-a, 2-2-b, 2-2-d, 2-2-e
 and 3-1, we set
%
$$  p_1(x;x^*) = \frac12 \sum_{j=1}^3 C_j s_j (x_j-x^*_j)^2; $$
%
 in Case 2-2-c, we set
%
$$  p_1(x;x^*) = \frac12 \sum_{j=1}^2 C_j(x_j-x^*_j)^2  - \frac14 C_3 (x_3-x^*_3)^4. $$
%
 Then, Relations\refq{prop1.dV} and\refq{prop2.dV} of Lemma \ref{lem.dV} can be written
%
$$  \rmd V(x) = (1+O(\rmd_0(x,x^*))) \rmd p_1(x;x^*), \quad x\to x^*. $$
%

\subsubsection{The conjugated operator near thresholds}
\label{ssec.near}
 Let $x^*\in \calX^*$. For simplicity we then write $p_1(x;x^*)=p_1(x)$.
 For $u\in \calD_0$ and $x\in \T^3\sauf\{x^*\}$ we set,\\
 - in Case 1 ($\bfbeta=0$, $x^*\in \calX_2^*$):
%
 $$
  \calA_{x^*} u\: (x) := 
   i \chi_{x^*}(x) \pi_2(y) \frac{\nabla_x p_1(x)}  {\nabla_x p_1(x) \cdot \nabla_x \sqrt{\Psi_0(z)}}
   \nabla_x (\chi_{x^*}(x) \pi_2(y) u(x)),
$$
%
 - in Case 2 ($\bfbeta\neq 0$, $x^*\in \calX_1^{*\pm}$) and Case 3-1) ($\bfbeta\neq 0$, $\beta_2=0$,
 $x^*\in \calX_2^*$): 
%
$$
 \calA_{x^*}^\pm u\: (x) := 
   i \chi_{x^*}(x) \pi_1^\pm(y) \frac{\nabla_x p_1(x)}
  {\nabla_x p_1(x) \cdot \nabla_x \sqrt{\tau^\pm(z)}} \nabla_x (\chi_{x^*}(x) \pi_1^\pm(y) u(x)),
$$
%
 and $\calA_{x^*}:= \calA_{x^*}^+ +\calA_{x^*}^-$,

 - in Cases 3-2 and 3-3 ($\beta_2\neq0$, $x^*\in \calX_2^*$):
%
$$  \calA_{x^*} u\: (x) := i  \chi_{x^*}(x) \pi_2(y) (\sqrt{\Psi_0(z)}_w)^{-1} (\chi_{x^*}(x)\pi_2(y) u(x))_w . $$
%
 In each case we symmetrize $\calA_{x^*}$ and $\calA_{x^*}^\pm$ by setting
$$  A_{x^*} :=\calA_{x^*} +  \calA_{x^*}\!^* , \quad A_{x^*}^\pm :=\calA_{x^*}^\pm +  (\calA_{x^*}^\pm)^*,$$
 where $\calA_{x^*}\!^*$ (respect., $(\calA_{x^*}^\pm)^*$) denotes the formal adjoint to $\calA_{x^*}$
 (respect., to $\calA_{x^*}^\pm$). It is defined on $\calD_0$ too.\\

 We set
\begin{eqnarray*}
\calT_{in} &:=& (\calT\sauf \calT')\cap (0,+\infty),\\
  \calX_{1,in}^{*\pm} &:=& \{x\in \calX^{*\pm}_1; \sqrt{\tau^\pm(z)} \in \calT_{in}\} ,\\
  \calX_{2,in}^{*} &:=& \{x\in \calX_2^*; \sqrt{\Psi_0(z)} \in \calT_{in}\}.
\end{eqnarray*}
 (We have, in Case 1, $\calX_{1,in}^{*\pm} =\emptyset$.)
 We set
$$
 A_{in}:=\sum_{x^*\in  \calX_{2,in}^* } A_{x^*} + \sum_\pm \sum_{x^*\in  \calX_{1,in}^{*\pm} } A_{x^*}^\pm.
$$
 Then the operator $A_{in}$ with domain $\calD_0$ is symmetric, closable
 and densely defined on $\calHD$.

 We set, as quadratic forms defined on $\calD_0$,
$$ H_{1,x^*} := [H^D,iA_{x^*}] , \quad H_{1,in} := [H^D,iA_{in}] .$$
 By a straight calculation as in Part \ref{sec.Hout} we obtain
\begin{lemma}
\label{val.H1x*}
 We have for $x\neq x^*$, in Cases 1 and 2:
%
$$   \frac12 H_{1,x^*}(x)  = (\chi_{x^*}(x))^2 \pi_2(y), $$
%
 and, in Cases 3-2 and 3-3,
$$
   \frac12 H_{1,x^*}(x) = (\chi_{x^*}(x))^2 \pi_2(y)
    + (\chi_{x^*}(x))^2 \pi_2(y) (\sqrt{\Psi_0(z)}_w)^{-1}\; \xi_w(x)\pi_2(y).
$$
%
\end{lemma}
 Lemma \ref{val.H1x*} shows  that the quadratic forms $H_{1,x^*}$ and $H_{1,in}$
 extend continuously as bounded quadratic forms on $\calHD$ which are associated with bounded
 self-adjoint operators, as multiplication operators by smooth real symmetric coefficients, denoted,
 respectively, $H_{1,x^*}$ and $H_{1,in}$. 
 In addition, these coefficients (as functions of $x$) are commuting with $\pi_2(y)$.
 Then, an obvious iteration shows that  $H^D\in \calC^\infty(A_{x^*})$ for all $x^*\in \calX^*$.
 Since  $x\neq x'$ implies $\supp \chi_{x}\cap\supp \chi_{x'}=\emptyset$ then $H^D\in \calC^\infty(A_{in})$. 

 We set
%
$$ A_{\phi} : = A_{out}+A_{in}. $$
%
 The argumentation to prove the property $H^D\in \calC^\infty(A_{out})$ at Part \ref{ssc.smooth1}
 still holds with $A_{out}$ replaced by $A_{\phi}$, so we obtain
\begin{lemma}
\label{lem3.Hsmooth}
 The quadratic form $H_{1,\phi}:= [H^D,iA_{\phi}]$ defined on $\calD_0$ defines a bounded
 self-adjoint multiplication operator on $\calHD$.
 In addition, $H^D \in \calC^\infty(A_\phi)$.
\end{lemma}
\subsection{"Punctual" Mourre's estimate}
 Let us prove that
\begin{equation}
 \phi(H^D)(x)H_{1,\phi}(x)\phi(H^D)(x)  \ge  C \phi^2(H^D)(x) ,
\label{min.phiH1phi}
\end{equation}  
 for all $x\in\T^3$, where $C>0$ does not depend on $x$ but on $\phi$ only.

 We consider the case $\beta_2\neq0$ (under assumption (A0)) only. 
 The other case $\beta_2\neq0$ is more simple and omitted.
 As in Part \ref{sec.Hout} (see\refq{valfin3.HDAb}) the calculation of $H_{1,\phi}$ yields
\begin{eqnarray*} 
 \frac12 H_{1,\phi}(x) & = & \sum_\pm (\chi^{*\pm}_1(x))^2 \pi_1^\pm(y) 
  + ((\chi^*_2(x))^2 + (\chi^*_3(x))^2) \pi_2(y)  \\ 
 &&  + (\chi^*_2(x))^2  \pi_2(y) T_w(x)^{-1}\; \xi_w(x)\pi_2(y) \\ 
 && + \sum_\pm \sum_{x^*\in \calX_{1,in}^{*\pm} \cup   \calX_{1,in}^{*-}} (\chi_{x^*}(x))^2  \pi_1^\pm(y) 
 + \sum_{x^*\in \calX_{2,in}^{*} }  (\chi_{x^*}(x))^2  \pi_2(y)  \\
 && + \sum_{x^*\in \calX_{2,in}^{*}} (\chi_{x^*}(x))^2  \pi_2(y) (\sqrt{\Psi_0(z)}_w)^{-1}\; \xi_w(x)\pi_2(y) 
  \quad x\not\in\calX^*\cup \calX_0.
\end{eqnarray*}
 Thus, as for Inequality\refq{min3.H1out}, we get,
\begin{eqnarray}
\nonumber 
 H_{1,\phi}(x) & \ge & \sum_\pm (\chi^{*\pm}_1(x))^2 \pi_1^\pm(y)  + ((\chi^*_2(x))^2 + (\chi^*_3(x))^2) \pi_2(y)  \\
\nonumber 
 && + \sum_\pm \sum_{x^*\in \calX_{1,in}^{*\pm} \cup   \calX_{1,in}^{*-}} (\chi_{x^*}(x))^2  \pi_1^\pm(y) 
 + \sum_{x^*\in \calX_{2,in}^{*} }  (\chi_{x^*}(x))^2  \pi_2(y) .\\
&& 
\label{min3.H1phi}
\end{eqnarray}  
 Let us fix $x\in \supp\phi(H^D(x))$. Thus $\chi_0(x)=1$. We consider the following cases.\\
\begin{enumerate}
\item
Case $\rmd_0(x,\calX^*)\ge b$. Then, $x\not\in \supp \chi_{x^*}$ for any $x^*\in \calX^*$,
 and $\chi^*(x)= \chi^{*\pm}(x)=1$. Hence\refq{min3.H1phi} becomes
\begin{eqnarray*} 
 H_{1,\phi}(x) & \ge & \sum_\pm \chi_1^2(x) \pi_1^\pm(y)  + (\chi_2^2(x) + \chi_3^2(x)) \pi_2(y) 
 =\sum_{j=1}^3\chi_j^2(x) \pi_2(y) \\
 & \ge & \delta_0 \pi_2(y)  ,
\end{eqnarray*}
 where $\delta_0:=\min_{\T^3}\sum_{j=1} \chi_j^2 >0$.
 Since $\phi(H^D(x))\pi_2(x)=\phi(H^D(x))$, then\refq{min.phiH1phi} holds.

\item
 Case $\rmd_0(x,\calX^*)< b$. Then there exists exactly one $x^*\in\calX^*$
 such that $x\in \supp \chi_{x^*}$ and $x\not\in \supp \chi_{x'}$ if $x'\in \calX^*\sauf\{x^*\}$.
 We set 
$$ \delta(x^*):=\min_{\{\chi_0=1\}}(1-\chi_{x^*})^2 + (\chi_{x^*})^2>0 .$$
 If $x^*\not\in \calX^*_{2,in}\cup\calX^{*+}_{1,in}\cup\calX^{*-}_{1,in}$ then $x\not\in\sup\phi(H^D)(x)$
 so\refq{min.phiH1phi} is trivial. We thus assume $x^*\in \calX^*_{2,in}\cup\calX^{*+}_{1,in}\cup\calX^{*-}_{1,in}$.
\begin{enumerate}
\item
 Case $x^*\in \calX^*_{2,in}$. Thus $x\not\in\supp \chi_1\cup\supp\chi_3$ and
$$
 \chi^*_2(x) = (1-\chi_{x^*}(x)) \chi_2(x)=(1-\chi_{x^*}(x)) .
$$
 Hence\refq{min3.H1phi} becomes
$$ H_{1,\phi}(x)  \ge (\chi^*_2(x))^2 \pi_2(y) + (\chi_{x^*}(x))^2 \pi_2(y)   \ge \delta(x^*) . $$
%
 Thus\refq{min.phiH1phi} holds.

\item Case $x^*\in \calX^{*+}_{1,in}$ (which is included in $\calX^{*-}_{1,in}$
 and does not intersect $\calX^*_{2,in}$).
 Thus $x\not\in\supp \chi_2\cup\supp\chi_3$ and
$$
 \chi^{*\pm}_1(x) = (1-\chi_{x^*}(x)) \chi_1(x)= 1-\chi_{x^*}(x) .
$$
 Hence\refq{min3.H1phi} becomes
\begin{eqnarray*}
 H_{1,\phi}(x) & \ge & \sum_\pm ((1-\chi_{x^*}(x))^2 + (\chi_{x^*}(x))^2 ) \pi_1^\pm(y) \\
 &&  = ((1-\chi_{x^*}(x))^2  + (\chi_{x^*}(x))^2) \pi_2(y)\\
 & \ge & \delta(x^*) \pi_2(y).
\end{eqnarray*}
 Thus\refq{min.phiH1phi} holds.
 
\item Case $x^*\in \calX^{*-}_{1,in}\sauf\calX^{*+}_{1,in}$ (which does not intersect $\calX^*_{2,in}$).
 Thus $x\not\in\supp \chi_2\cup\supp\chi_3$ and
$$
 \chi^{*-}_1(x) = (1-\chi_{x^*}(x)) \chi_1(x)= 1-\chi_{x^*}(x) ,\quad
 \chi^{*+}_1(x)=0
 $$
 Hence\refq{min3.H1phi} becomes
\begin{eqnarray*}
 H_{1,\phi}(x) & \ge & \sum_\pm ((1-\chi_{x^*}(x))^2  + (\chi_{x^*}(x))^2) \pi_1^-(y) \\
 & \ge & \delta(x^*) \pi_1^-(y) .
\end{eqnarray*}
 But we have also
$$ \phi(H^D)(x)=\phi(\sqrt{\tau^-(x)})\pi_1^-(y). $$
 Thus\refq{min.phiH1phi} holds.
\end{enumerate}
\end{enumerate}
 As conclusion,\refq{min.phiH1phi} is proved with $C=\min(\delta_0,\min_{\calX^*}\delta(x^*))$.
 \qed

\subsection{Self-adjointness and maximal monotonicity of parts of the conjugated operator}
\label{sec.selfa}
 The conjugate operator $A_\phi$ with domain $\calD_0$ is a symmetric first order differential operator
 in $x$ whose coefficients belong to $\calC^\infty(\T^3\sauf\calX^*;\calL(\C^6))$.
 As we already saw, $A_{out}$ with domain $\calC^\infty(\T^3)$ is essentially self-adjoint
 and a self-adjoint extension is $\ove{A_{out}}=A_{out}$ with domain
 $D(\ove{A_{out}}) =\{u\in \calHD$; $A_{out}u\in \calHD\}$.
 (We may observe that $D(\ove{A_{out}})$ is also the closure of $\calC^\infty(\T^3)$ under the norm
 $\|u\|+\|A_{out} u\|$.)
 Let us check that for all $x^*\in \calX^*$ the operator $A_{x^*}$ (with the same domain $\calD_0$)
 is essentially self-adjoint or, at least, admits a maximal symmetric extension.
\begin{lemma}
\label{lem1.Aadjoint}
 Remembering the notations of Section \ref{sec.behave} we then claim:
\begin{itemize}
\item[A)] Cases 1 and 2-1 and 2-2-a and 2-2-b and 3-1.
 If $\{s_1,s_2,s_3\}=\{-1,1\}$, then $A_{x^*}$ is essentially self-adjoint on $\calH^D$.
 Otherwise, i.e., if all the $s_j$'s have the same sign, then $A_{x^*}$
 admits a maximal symmetric extension on $\calH^D$.
\item [B)] Case 2-2-c.  The operator $A_{x^*}$ is essentially self-adjoint on $\calH^D$.
\item [C)] Cases 3-2 and 3-3 (so we have (A0) with $\beta_2\in (0,\beta_1]$, $x^*\in \calX_2^*$,
 $|y^*_2|=1$).
 The operator $A_{x^*}$ admits a maximal symmetric extension on $\calH^D$.
\end{itemize}  
\end{lemma}
 Proof in Appendix B.
\begin{remark}
\label{rem.Ain-sa}
 When $A_{x^*}$ is essentially self-adjoint then the set $\calD_0$ is not dense
 in the domain of the self-adjoint extension $\ove{A_{x^*}}$ of $A_{x^*}$. 
 (The simple reason is that $\calC^\infty_c(\R^*)$ is not dense in $H^1(\R)$.)
\end{remark}
 We set
$$ 
 \calX^*_{sa}:=\{x^*\in \calX^*;\; \mbox{$A_{x^*}$ is essentially self-adjoint} \} ,
 \quad \calX^*_{sm} := \calX^* \sauf \calX^*_{sa},
$$
 and
$$ A_{sa}:=\sum_{x^*\in \calX^*_{sa}} A_{x^*}, \: D(A_{sa}):=\calD_0 .$$
\begin{coro}
\label{coro.Aphi-sa}
 1) The operator $A_{sa}$ is essentially self-adjoint on $\calHD$.\\
 2) If $\calT_{sm}\subset \calT'$ then the operator $A_\phi$ defined on $\calD_0$
  is essentially self-adjoint on $\calHD$.
\end{coro}
\begin{proof}
1) The operator $A_{sa}$ is the finite sum of essentially self-adjoint operators $A_{x^*}$
 defined on $\calD_0$ and with disjoint supports so $A_{sa}$ is essentially self-adjoint too.\\
 2) For simplicity we assume that $\calT'=\calT_{sm}$ and we consider the case $\beta_2\neq 0$ only.
 The operator $\bar A:=A_\phi$ with domain $D(\bar A):=\{u\in \calHD$; $A_\phi u\in \calHD\}$
 is a symmetric extension of $A_\phi$. Let us prove that it is self-adjoint.
 Let $v\in D(\bar A \,^*)$ so
$$ |(\bar A u,v)_{\calHD}|\le C\|u\| \quad \forall u\in D(\bar A). $$
 Let $u\in D(\bar A)$. Let $\varphi_1\in\calC^\infty(\T^3\sauf\calX^*_{sm};[0,1])$
 be such that $\supp\varphi_1$ is a small neighbourhood of $\calX^*_{sa}$ and
 $\varphi_1=1$ near $\calX^*_{sa}$.
 Setting $B:=\varphi_1 A_\phi  \varphi_1$, since $\nabla\varphi_1$ vanishes near
 $\calX^*$ then $B - A_\phi \varphi_1^2$ is bounded on $\calHD$, 
 $\varphi_1^2 u\in D(\bar A)$ and we get
$$  |(Bu,v)_{\calHD} | \le |(\bar A (\varphi_1^2 u),v)_{\calHD} | + C'\|u\|  \le  C''\|u\|. $$
 In addition,  we have $B=\varphi_1 A_{sa} \varphi_1$ since $\bar A$ coincides with $A_{sa}$
 in $\supp\varphi_1$, so $B$ is essentially self-adjoint (the proof is similar to those of $A_{sa}$).
 Hence, $Bv \in \calHD$ and then $\varphi_1^2 v\in D(\bar A)$. 
 Let $\varphi_2 \in\calC^\infty_c(\T^3\sauf \calX^*; [0,1])$.
 Then, $\varphi_2 A_\phi \varphi_2 - A_\phi\varphi_2^2$ is bounded on $\calHD$,
 $\varphi_2^2 u \in D(\bar A)$ and
$$ |(\varphi_2 \bar A (\varphi_2 u),v)_{\calHD} |   \le  |(A (\varphi_2^2 u), v)_{\calHD} |+C'\|u\|  \le  C''\|u\|. $$
 Since $\varphi_2 A_\phi \varphi_2$ is a symmetric first order differential operator with smooth coefficients
 it is so essentially self-adjoint and we get $\varphi_2 A_\phi \varphi_2 v\in \calHD$, and
 $\varphi_2^2 v\in D(\bar A)$.
 Letting $\varphi_1$ such that its derivatives at any order vanish on $\varphi_1^{-1}(\{1\})$
 we can choose $\varphi_2:=\sqrt{1-\varphi_1^2}$. Then $v=\sum_{j=1}^2\varphi_j^2 v\in D(\bar A)$.
\end{proof}
\section{Proofs of the main results}
\label{sec.6}
\subsection{Proof of Theorem \ref{th.1}}
 Clearly, it is not restrictive to consider that $\supp\phi\subset(0,+\infty)$ so
 $\phi\in\calC^\infty_c((0,+\infty)\sauf\calT')$.
 We then construct the operators $A_{out}$, $A_{in}$, $A_\phi$ as above.
 Thanks to Lemma \ref{lem3.Hsmooth}, the operator $A_{\phi}$ satisfies Point (ii). 
 Point (i) is a straight consequence of\refq{min.phiH1phi}.

 Proof of Point (iii). We consider the cases $\supp\phi\subset(0,+\infty)$ and $\beta_2\neq 0$
 only. We have $\supp A_\phi\subset \calX_A$ where we set
$$
\calX_A:=\supp \chi_1^{*+}\cup \supp \chi_1^{*-} \cup \supp \chi_2^* \cup \supp \chi_3^*
\cup_{x^*\in \calX_{1,in}^{*+}\cup \calX_{1,in}^{*-} \cup\calX_{2,in}^*} \supp \chi_{x^*}  . 
$$
 The set $\calK:=\cup_\pm\sqrt{\tau^\pm(\sin^2(\calX_A))}$ is then a compact subset of $(0,\infty)\sauf\calT'$.
 Thus there exists $\tilde\phi\in \C^\infty_c((0,\infty)\sauf\calT')$ with $\tilde\phi=1$ in $\calK$. 
 Thus if $x\in\supp A_\phi\cap \calX_1$ then 
$$ \tilde\phi(H^D)(x)= \sum_\pm\tilde\phi(\sqrt{\tau^\pm(z})) \pi_1^\pm(y)= \pi_2(y),$$
 and, if $x\in\supp A_\phi\cap \calX_2$, then 
$$ \tilde\phi(H^D)(x)= \tilde\phi(\sqrt{\Psi_0(z})) \pi_2(y)= \pi_2(y).$$
 Hence, $\tilde\phi(H^D)(x)=\pi_2(y)$ on $\supp A_\phi$.
 In addition, $A_\phi(x)$ is obviously commuting with $\pi_2(y)$ for all $x$.
 It shows that Point (iii) holds.

 Point (iv) is Point 2) of Corollary \ref{coro.Aphi-sa}.
 
 Point (v) is the consequence of Corollary \ref{coro.Aphi-sa}.

\subsection{Adaptation of the theory of Georgescu and alii}
\label{sec.ms-adapt}
 Notation: if $Q$ is a bounded quadratic form on $\calH$ we denote by $Q^\circ$ the bounded operator
 associated with $Q$.
 Let us consider the case $\calT'=\{0\}$ so $A_{in}$ may not be essentially self-adjoint.
 We set
$$  \calX^*_{sm1} := \{x^*\in  \calX^*_{sm}; \: A_{x^*}\mbox{ has default index $(N^+,N^-=0)$}\}, $$
$$  \calX^*_{sm2} := \{x^*\in  \calX^*_{sm}; \: A_{x^*}\mbox{ has default index $(N^+=0,N^-)$}\}. $$
 We write
$$ A_\phi = A_0+A_1+A_2 $$
 where all the $A_j$ are differential operators of first order defined at least on $\calD_0$ by:
 $A_0=A_{out}$, $A_1=\sum_{x^*\in \calX^*_{sm1}\cup  \calX^*_{sa}}A_{x^*}$,
 $A_2=\sum_{x^*\in \calX^*_{sm2}}A_{x^*}$.
 The proof of Corollary \ref{coro.Aphi-sa} shows that the operator $A_0$ is essentially self-adjoint.
\begin{remark} 
 We could have set more naturally $A_0=A_{sa}+A_{out}$, $A_1=\sum_{x^*\in \calX^*_{sm1}}$
 and $A_2$ unchanged. In such a choice some coefficients of $A_0$ have a rational singularity
 on $\calX^*_{sa}$.
\end{remark}
 Since the supports of the $A_{x^*}$, $x^*\in \calX^*$, are two-by-two disjoint then the operators $\pm A_0$
 and $(-1)^j A_j$, $j=1,2$, admit a maximal symmetric extension with deficiency index of the form $(N,0)$.
 We denote by $A_j^{sm}$ with domain $D(A_j^{sm})$ the maximal symmetric extension of $A_j$
 (with domain $\calD_0$).\\
 Let us show that we can modify the main hypotheses (M1)-(M5) of \cite[Theorem 3.3]{GEO.COM}
 and extend the statement of \cite[Theorem 3.3]{GEO.COM} to our situation.
 We consider variables $z\in\rho(H^D)$ and $\epsl$ real with $0<|\epsl|<\epsl_0$
 and $\Im m(z)\epsl\ge 0$.
 We set $H_\epsl := H^D-i\epsl H'$. (Thus $H_\epsl ^*=H_{-\epsl}$.)
 Then, the resolvent $R_\epsl(z):= (H_\epsl-z)^{-1}$ is well-defined if $\epsl_0$
  is sufficiently small, see \cite[Proposition 3.11]{GEO.COM}.
  Actually we make the following observations:
\begin{itemize}
\item The domain $\calD_0$ of $A_\phi$ is dense in $\calHD$.
\item Assumption \cite[(M3)]{GEO.COM} becomes:\\
 (M3*): [ $\pm A_0^{sm}$ (respect., $(-1)^j A_j^{sm}$) is the generator of a $C_0$-group $(W^{(0)}_t)_{t\in\R}$
 (respect., semigroup $(W^{(j)}_t)_{t\ge 0}$) in $\calHD$.]\\
 Clearly, Condition (M3*) is satisfied.
\item
 Setting $<H>:=(1+H^2)^{1/2}$, Assumption \cite[(M2)]{GEO.COM} becomes:\\
 (M2*): {[} a bounded open set $J\subset \R$ is given and there are numbers $a>0$, $b\ge 0$,
 such that $H' \ge (a1_J(H^D)-b1_{\R\sauf J})<H^D>$ as forms on $\calHD$.]\\
 Thus, for all bounded open set $J\subset\subset (0,+\infty)$, Condition (M2*) is satisfied
 (by choosing $\phi$ such that $\phi=1$ on $J$, and with $b=0$),
 thanks to Mourre's Inequality\refq{min.phiH1phi}.
\item
 Assumption \cite[(M4)]{GEO.COM} becomes:\\
 (M4*): {[} There is $H'_j\in B(\calHD)$ such that the limits
\begin{eqnarray*}
 \lim_{t\to 0^+} & (-1)^j t^{-1} \{(H^D u,W^{(j)}_tu)_{\calHD}-(u,W^{(j)}_t H^D u)_{\calHD}\} \quad (j\neq 0),\\
 \lim_{t\to 0} & t^{-1} \{(H^D u,W^{(0)}_tu)_{\calHD}-(u,W^{(0)}_t H^D u)_{\calHD}\} \quad (j=0)
\end{eqnarray*}
 exist and are respectively equal to $(u,H'_ju)$ for all $u\in \calHD$.]\\ 
 Clearly, Condition (M4*) is satisfied.
 We have $[H^D,iA_j]^\circ=H'_j$ and $[H^D,iA_\phi]^\circ=H' :=\sum_{j=0}^2 H'_j\in B(\calHD)$,
 and we can write $H^D\in C^1(A_j^{sm})$, $H^D\in C^1(\bar A_\phi)$.
\item The proofs of \cite[Lemmas 3.13 and 3.14]{GEO.COM} with Conditions (M3*) and (M4*) satisfied
 imply the following relations:
\begin{eqnarray*}
 {[}R_\epsl(z),iA_j^{sm}]^\circ &=& R_\epsl(z)(iH'_j+\epsl H''_j)R_\epsl(z) \quad j=0,1,2, \\
 {[}R_\epsl(z),iA_\phi]^\circ &=& R_\epsl(z)(iH'+\epsl H'')R_\epsl(z),\\
  \frac{\rmd R_\epsl(z)}{d\epsl} &=&  [R_\epsl(z),iA_\phi]^\circ - \epsl R_\epsl(z)H'' R_\epsl(z).
\end{eqnarray*}
 In particular the map $\epsl\mapsto R_\epsl(z) \in B(\calHD)$ is $C^1$ in norm on $]0,1]$.
\item
 Since $H^D$ and the $H'_j$'s are symmetric bounded self-adjoint operators on $\calHD$
 (so $H'_j$ is regular; see also \cite[Remark 2.15]{GEO.COM}), 
 then Assumption \cite[(M1)]{GEO.COM} becomes:\\
 (M1*): {[} for all $j$,  $H^D\in C^1(H'_j)$.]\\
 We see that (M1*) is obviously satisfied and $H^D\in C^1(H')$.
\item
 Assumption \cite[(M5)]{GEO.COM} becomes:\\
 (M5*): {[} for all $j=0,1,2$, there is $H''_j\in B(\calHD)$ such that the limits
\begin{eqnarray*}
 \lim_{t\to 0^+} & (-1)^j t^{-1} \{(H'u,W^{(j)}_tu)-(u,W^{(j)}_t H'u)\} \quad j\neq 0,\\
 \lim_{t\to 0} & t^{-1} \{(H'u,W^{(0)}_tu)-(u,W^{(0)}_t H'u)\} \quad(j= 0),
\end{eqnarray*}
 exist and are respectively equal to $(u,H''_ju)$ for all $u\in \calHD$.]\\
 Thanks to \cite[Remark 3.1]{GEO.COM}, Condition (M5*) is satisfied since it follows from
 the following facts:
 $H^D\in C^1(A_j^{sm})$ with $[H^D,iA_j^{sm}]^\circ = H'_j$, $H'\in C^1(A_j^{sm})$ with
 $[H',iA_j^{sm}]^\circ = H''_j$, so we can write $H^D\in C^2(\bar A_\phi)$.
\end{itemize}
 Then the proof of \cite[Theorem 3.3]{GEO.COM}  implies the result of Corollary \ref{coro.3}. \qed

\section*{Conclusion}
 The results of this work, notably the LAP outside thresholds, is the first step to futur
 developments as: 
\begin{itemize}
\item Extension of the result of Isozaki and Morioka \cite{ISO.REL} on
 Rellich type theorem for discrete Schrodinger operators to the case of
 discrete Maxwell operators.
\item The LAP for perturbed discrete Maxwell operator of the form $\hat H^D=\hat D \hat H_0$
 where $\hat D$ is not constant but depends on $n\in \Z^3$. 
\item Conditions of radiation for perturbed discrete Maxwell operators.
 Actually, let $\hat f$ be in a suitable subspace of $L^2(\Z^3)$, particularly the space of sequences
 with compact support.
 We have to characterize $\hat u^\pm(n)$ for $|n|$ large where $\hat u^\pm:=(\hat H^D-\lambda\pm i0)^{-1}\hat f$.
\item Extension of the result of Isozaki and Jensen \cite{ISO.CON} on
 the continuum limit for lattice Schr\"odinger operators to the case of discrete Maxwell operators.
\item Extension of the result of Isozaki and  \cite{ISO.INV} on the inverse scattering
 for lattice Schr\"odinger operators to the case of discrete Maxwell operators.
\end{itemize}
%

\section{Appendix A}
\label{sec.AppendixA}
\subsection{Proof of Lemma \ref{lem.valp}}
 We have
$$
 \hat D \hat H_0(x) -\lambda= \left( \begin{array}{cc} -\lambda & \epsl M \\ -\mu M & -\lambda \end{array} \right).
$$
 Thus,
$$ \det(\hat D \hat H_0(x) -\lambda) = \det(\lambda^2 + \epsl M(y) \mu M(y)) =: p(z;\lambda).$$
 We have
\begin{eqnarray*}
 \epsl M \mu M &=& \left( \begin{array}{ccc}
 0 & -\epsl_1 y_3  & \epsl_1 y_2\\
 \epsl_2 y_3  & 0 &  - \epsl_2 y_1\\
 -\epsl_3 y_2  &   \epsl_3 y_1 & 0 \end{array} \right) 
  \left( \begin{array}{ccc}
 0 & -\mu_1 y_3  & \mu_1 y_2\\ 
 \mu_2 y_3  & 0 &  - \mu_2 y_1\\
 -\mu_3 y_2  &  \mu_3 y_1 & 0
\end{array} \right)  \\
 &=& 
\left( \begin{array}{ccc}
 -\epsl_1 \mu_3 y_2^2  -\epsl_1 \mu_2 y_3^2 & \epsl_1 \mu_3 y_1y_2 &
 \epsl_1 \mu_2  y_1y_3 \\
  \epsl_2 \mu_3 y_1y_2 &   -\epsl_2 \mu_3 y_1^2 -\epsl_2 \mu_1 y_3^2 &
  \epsl_2 \mu_1 y_2y_3\\
  \epsl_3 \mu_2 y_1y_3 &  \epsl_3 \mu_1 y_2y_3 &
  -\epsl_3 \mu_2 y_1^2 -\epsl_3 \mu_1 y_2^2 
\end{array} \right).
\end{eqnarray*}
 Then, for $t=-\lambda^2\in\C$,
\begin{eqnarray*}
 \det (\epsl M \mu M - t) &=& -t^3 - t^2\{(
 (\epsl_2 \mu_3 + \epsl_3 \mu_2) y_1^2 +
 (\epsl_1 \mu_3 +\epsl_3 \mu_1 )y_2^2 +
 (\epsl_1 \mu_2 + \epsl_2 \mu_1)y_3^2\} \\
 && - t \{
  \epsl_2 \epsl_3  \mu_2\mu_3 y_1^4 
 + \epsl_1 \epsl_3  \mu_1\mu_3 y_2^4
 + \epsl_1 \epsl_2  \mu_1\mu_2 y_3^4
 + (\epsl_2 \epsl_3 \mu_1\mu_3 + \epsl_1 \epsl_3 \mu_2\mu_3)y_1^2y_2^2
\\
&&
 + (\epsl_2 \epsl_3 \mu_1\mu_2 + \epsl_1 \epsl_2 \mu_2\mu_3)y_1^2y_3^2
 + (\epsl_1 \epsl_3 \mu_1\mu_2 +  \epsl_1\epsl_2 \mu_1\mu_3)y_2^2y_3^2
 \}\\
 &\equiv & -t^3 - 2t^2 \Psi_0 -t \Phi_0,
\end{eqnarray*}
 where $\Psi_0(z)$ is defined by\refq{def.Psi0} 
%
$$
 \Phi_0 :=  \epsl_2 \epsl_3  \mu_2\mu_3 z_1^2  + (\epsl_2 \epsl_3 \mu_1\mu_3
   + \epsl_1 \epsl_3 \mu_2\mu_3)z_1 z_2 + c.p. .
$$
 We easily observe that the following relations hold (also, $+ c.p.$):
\begin{eqnarray}
\label{r.alphagamma}
 \alpha_1^2-\gamma_1 &=& \frac14 \beta_1^2,\\
\label{r.alphabeta}
 \epsl_1\mu_1 \alpha_1 - \alpha_2\alpha_3 &=& \frac14\beta_2\beta_3,\\
\label{r.betai}
 \alpha_3\beta_2 + \alpha_2 \beta_3 &=& -\epsl_1\mu_1 \beta_1,
\end{eqnarray}
 where $\gamma$ is defined by\refq{def.gamma} and $\bfbeta$ by\refq{def.beta}.
 Thanks to\refq{r.alphagamma},\refq{r.alphabeta},\refq{r.betai}, we compute:
\begin{eqnarray*}
 \Psi_0^2-\Phi_0 &=&
 (\alpha_1 z_1 + \alpha_2 z_2 + \alpha_3 z_3)^2/4 - \{  \gamma_1 z_1^2  + \gamma_2 z_2^2 \\
 &&  + \gamma_3 z_3^2
  + 2\epsl_3 \mu_3 \alpha_3 z_1z_2 + 2\epsl_1 \mu_1 \alpha_1 z_2z_3
  + 2\epsl_2 \mu_2 \alpha_2 z_1z_3\}/4 \\
  &=& K_0(z),
\end{eqnarray*}
 where $K_0$ is defined by\refq{def.K0}.
 Hence the eigenvalues of $\epsl M(y) \mu M(y)$ are 0 and
$$ t=k^2 = \Psi_0(z) \pm \sqrt{K_0(z)}.$$
 Then Relation\refq{v.taupm} follows.

\begin{remark}
 From\refq{def.K0} we obtain the relations
\begin{eqnarray}
\label{val1.K0}
 K_0(z)  &=& \frac14(\beta_1 z_1- \beta_2 z_2 - \beta_3 z_3)^2 - \beta_2\beta_3 z_2z_3,\\
\nonumber
 K_0(z)  &=& \frac14(\beta_3 z_3- \beta_1 z_1 - \beta_2 z_2)^2 - \beta_1\beta_2 z_1z_2.
\end{eqnarray}
\end{remark}

\subsection{Proof of Proposition \ref{p.HD-spectre}}
 Similarly to\refq{val1.sigma} we have
$$ \sigma(h^D) = \ove{\cup_{y\in [-1,1]^3} \sigma(h^D(y))}.$$
 In addition, since $\tau^+$ and $\tau^-$ are continuous with $\tau^\pm(0)=0$
 and $\tau^+\ge\tau^-\ge 0$ in $[0,1]^3$, then
$$
 \cup_{y\in [-1,1]^3}\sigma(h^D(y)) = \cup_{\pm} \{\pm \tau^+(z); \; z\in [0,1]^3\} 
= [-\textcolor{red}{\lambda_+},\textcolor{red}{\lambda_+}].
$$
 The conclusion follows.
\subsection{Proof of Lemma \ref{lem.dtau=0}}
 We set $z'_i= \beta_i z_i$ (so $z'_1,z'_2\ge 0$ and $z'_3\le 0$).
 Remember that
$$
 (4\alpha_1^2-\beta_1^2) = (2\alpha_1-\beta_1) (2\alpha_1+\beta_1)
=  4\epsl_3\mu_2\epsl_2\mu_3 = 4\gamma_1 >0.
$$ 
\begin{itemize}
\item [A)] (Case $\bfbeta=0$). This point is obvious.
\item [B)] (Case $\bfbeta\neq0$.) Thanks to\refq{val1.K0} we have
\begin{eqnarray}
\label{val.dz1K0}
  \frac{\partial}{\partial {z_1}} K_0 &=&  \frac12 \beta_1 (z'_1 -  z'_2 - z'_3), \\
\nonumber  \frac{\partial}{\partial {z_3}} K_0 &=& \frac12 \beta_3 (z'_3 -  z'_1 - z'_2),
\end{eqnarray}
 and
\begin{equation}
\label{ine1.K0}
  \sqrt{K_0(z)} \ge  \frac12 |z'_1 -  z'_2 - z'_3|.
\end{equation}
 \begin{itemize}
 \item[1)] We have 
$$  \sqrt{K_0} \nabla_z\tau^\pm = \sqrt{K_0} \nabla_z\Psi_0 - \frac12\nabla_z K_0, $$
 so, by using\refq{val.dz1K0},\refq{ine1.K0},
\begin{eqnarray*}
 2\sqrt{K_0(z)}\partial_{z_1}\tau^\pm(z) &=& 2\sqrt{K_0(z)} (\alpha_1 - \frac12 |\beta_1|) 
 +  |\beta_1| \sqrt{K_0(z)} \pm \frac12\beta_1 |z'_1- z'_2 - z'_3| \\
 &\ge& 2\sqrt{K_0(z)} (\alpha_1 - \frac12 |\beta_1|) > 0.
\end{eqnarray*}
 Hence $\partial_{z_1} \tau^\pm(z)> 0$. Similarly, $\partial_{z_2}\tau^\pm(z) >0$.\\
 We have
$$
 2\sqrt{K_0} \partial_{z_3} \tau^\pm(z) =
 2\sqrt{K_0} \alpha_3  \pm \frac12 \beta_3 (z'_3- z'_1 - z'_2).
$$
 Since $\beta_3<0$ and $z'_3- z'_1 - z'_2\le 0$ then $\partial_{z_3} \tau^+(z)\ge 0$.
 Moreover if $\partial_{z_3} \tau^+(z)= 0$ then $\sqrt{K_0(z)}=0$ which is forbidden. 
 Hence $\partial_{z_3} \tau^+(z)> 0$.
 \item[2)]
\begin{itemize}
\item [a)] (Case $\beta_2=0$.) Thanks to\refq{val1.K0}, we have
$$
 2\sqrt{K_0} \partial_{z_3} \tau^-(z) = 
  \alpha_3 (z'_1- z'_3)   - \frac12 \beta_3 (z'_3- z'_1)
  = 2\sqrt{K_0} (\alpha_3- \frac12 \beta_3)>0.
$$
\item [b)] (Case $\beta_2>0$.)
 \begin{itemize}
\item [(i)] 
  Assume $z_1=0$ or $z_2=0$.
 Then $\sqrt{K_0(z)}= \frac12 (z'_1 + z'_2 -z'_3)$ and
$$ 2\sqrt{K_0} \partial_{z_3} \tau^-(z) = (\alpha_3 - \frac12|\beta_3| ) (z'_1 + z'_2 -z'_3)>0,$$
 so, $\partial_{z_3} \tau^-(z) >0$.
\item [(ii)] 
 Assume $z_1=z_2=1$.
 The functions $\xi:= 4K_0(z) \partial_{z_3}\tau^-(z) \partial_{z_3}\tau^+(z)$ and
 $\partial_{z_3}\tau^-(z)$ have the same sign outside $K_0^{-1}(\{0\})$. We have
\begin{eqnarray*}
 \xi  &=& 4\alpha_3^2 K_0(z) - (\partial_{z_3} K_0(z))^2 \\
 &=&  4\alpha_3^2 (\frac14(\beta_1 + \beta_2 - z'_3)^2 - \beta_1 \beta_2)
  - \frac14 \beta_3^2(\beta_1 + \beta_2 - z'_3)^2 \\
 &=& \gamma_3 (\beta_1 + \beta_2 - z'_3)^2  - 4 \alpha_3^2 \beta_1 \beta_2.
\end{eqnarray*}
 Thus, if $\gamma_3=0$ then $\xi <0$ and if $\gamma_3\neq 0$ then
\begin{eqnarray*}
 \xi=0  &\egal&  \sqrt{\gamma_3}(\beta_1 + \beta_2 - \beta_3 z_3)
  = 2\alpha_3 \sqrt{\beta_1\beta_2}\\
 &\egal&  z_3 = \nu.
\end{eqnarray*}
 Moreover, we have
$$ \partial_{z_3}\xi=0 = -2 \sqrt{\gamma_3}\beta_3(\beta_1 + \beta_2 - \beta_3 z_3) >0. $$
 The conclusion follows. 
  \end{itemize}
 \end{itemize}
\end{itemize}
\end{itemize}

\subsection{Proof of Lemma \ref{lem.seuils}}
 Lemma \ref{lem.seuils} is a straightforward consequence of Lemma \ref{lem.dtau=0} and
 of the following observations.
 
 1) Case $\bfbeta=0$. We have $K_0\equiv 0$ so $\calX_1=\emptyset$, $\calT_1=\emptyset$,
 $\calX_2=\T^3\sauf\{0\}$.
 In addition we have $\partial_{x_i}\Psi_0= 2\alpha_i \sin x_i\cos x_i$.
 Hence $\nabla_x \Psi_0(z)$ vanishes iff $z\in \{0,1\}^3$.
 Hence, noting that $\tau^\pm=\Psi_0$, we obtain $Z_2^*=Z_{\{0,1\}}^*$ and $\calT_2=\Psi_0(Z_{\{0,1\}}^*)$.

 2) Case $\bfbeta\neq0$ (so (A0) holds). 
 Thanks to Lemma \ref{lem.K0=0} we have 
$$
   \sin^2(\calX^*_1) \subset  [0,1]^3\sauf \{(\frac{\beta_2}{\beta_1}t, t,0), t\in [0,1] \} .
$$
 For $t\in [0,1]$ we have $(\beta_2 t/\beta_1, t,0) \in \{0,1\}^3 \sauf\{0_{\R^3}\}$ iff  $t=1$
 and $\beta_2=\beta_1$, or $t=1$ and $\beta_2=0$.
 The characterization of $ \sin^2(\calX^*_1)$ follows, then those of $\calX_1^*$ and of $\calT_1$.
 Let us determine $\calX_2^*$. 
 We look for a tangent vector field to $\calX_2$.
 A point $x\in\T^3$ belongs to $\calX_2$ iff $z\neq 0$ and $z_3=0=\beta_1 z_1 - \beta_2 z_2=0$.
 The last relation can be written
%
$$ \beta_1 y_1^2 = \beta_2 y_2^2, $$
%
 and $y_2\neq 0$. (If $y_2=0$ then $y_1=0$ so $z=0$ which is forbidden.)
 Then a tangent field to $\calX_2$ (respect., to $\sin(\calX_2)$) is then given by
 the vector field
\begin{equation}
\label{def.w0}
 w_0(x) := (\sin(x_1)\cos(x_2),\cos(x_1)\sin(x_2),0),
\end{equation}
(respect.,
$$ \tilde w_0(y) := (y_1,y_2,0) $$
).
\begin{remark}
 In Case 3-1) (i.e, $\beta_2=0$) it is equivalent but more simple to set $w_0(x) = (0,1,0)$.
 However our choice in\refq{def.w0} is general.
\end{remark}
 We then observe that $|w(x)|\neq 0$ for all $x\in \calX_2\sauf\calX^*$,
 and $|w(x)|\neq 0$ for all $x\in \calX_2$ if $\beta_2\in [0,\beta_1)$.
 If $\beta_2=\beta_1$ then $w(x)$ vanishes at all $x^*\in \calX_2^*$ since $z^*_1=z^*_2=1$.
 The determination of $\sin^2(\calX^*_2)$ follows, then those of $\calX_2^*$ and of $\calT_2$.

\subsection{Proof of the statement of Remark \ref{rem.nu}}
 Step 1. We prove the following assertion.
 Let $\bfbeta\in\R^3$ be a real vector and $\epsl_1$, $\mu_1$, $\alpha_3$ be three positive real
 values such that $\beta_1\ge \beta_2>0>\beta_3$ and $\alpha_3>|\beta_3|/2$.
 Then there exists positive values $\epsl_j$, $\mu_j$, $j=2,3$, such that
 $\bfbeta=\bfepsl\times\bfmu$ and $2\alpha_3=\epsl_1\mu_2+\epsl_1\mu_2$.

 Proof. We set successively
\begin{eqnarray*}
 \delta^{\pm} &:=& \alpha_3 \pm \frac12 \beta_3 >0,\\
 \epsl_2 &: =& \frac{\delta^{-}}{\mu_1} >0,\\
 \mu_2 &: =& \frac{\delta^{+}}{\epsl_1}>0 ,\\
 \epsl_3 &: =& \frac{\epsl_1\beta_1 + \epsl_2\beta_2 }{-\beta_3}>0 ,\\
 \mu_3 &: =& \frac{\mu_1\beta_1 + \mu_2\beta_2 }{-\beta_3}>0 .
\end{eqnarray*}
 A direct calculation provides $\bfepsl\times\bfmu=\bfbeta$ and $\epsl_1\mu_2+\epsl_2\mu_1=2\alpha_3$. \qed

 Step 2.
 Let $\bfbeta\in\R^3$ be such that $\beta_1\ge \beta_2>0>\beta_3$.
 Let us consider the function $(|\beta_3|/2,+\infty) \ni \alpha_3\mapsto \nu(\alpha_3)=\nu$
 defined by\refq{def.nu} and set
$$ F(r): = \frac{2r}{\sqrt{r^2-\frac14\beta_3^2}} - 2 \quad {\rm for} \: r>|\beta_3|/2,$$
$$
 \nu^*(\bfbeta) := - \frac{(\sqrt{\beta_1} -\sqrt{\beta_2})^2}{|\beta_3|}.
$$
 Then we have
$$
 \nu(\alpha_3) = \tilde\nu \egal  F(\alpha_3) =
  \frac{ |\beta_3| }{\sqrt{\beta_1\beta_2}} (\tilde\nu -\nu^*(\bfbeta) ).
$$
 Obviously, the function $F$ realizes a decreasing bijection from $(|\beta_3|/2,+\infty)$ into $(0,+\infty)$.
 Thus if $\nu^*(\bfbeta)<\tilde\nu$ then there exists a unique value $\alpha_3>|\beta_3|/2$ such
 that $\nu(\alpha_3)=\tilde\nu$. But the condition $\nu^*(\bfbeta)<\tilde\nu$ is easily satisfied since
 $\nu^*(\bfbeta)\to -\infty$ as $\beta_3\to 0^-$ if $\beta_1>\beta_2$. The conclusion follows.
 
\section{Appendix B}
\label{sec.AppendixB}
\subsection{Proof of Lemma \ref{lem.dV}}
 Firstly we observe that if $z^*_j\in \{0,1\}$ then
\begin{eqnarray}
\nonumber
 z_j-z^*_j &=& \sin^2 x_j-\sin^2 x^*_j = \sin(2x^*_j)(x_j-x^*_j) + 2\cos(2x^*_j) (x_j-x^*_j)^2 \\
\nonumber
  && + O( (x_j-x^*_j)^3)\\
\label{val1.z-z*}
 &=& 2s^*_j (x_j-x^*_j)^2 + O( (x_j-x^*_j)^4),
\end{eqnarray}
 and if $z_j\not\in \{0,1\}$ then
\begin{equation}
\label{val0.z-z*}
 z_j-z^*_j = \sin(2x^*_j)(x_j-x^*_j) + O( (x_j-x^*_j)^2) ,
\end{equation}
 with $\sin(2x^*_j)\neq 0$.

 Case 1. Assume $\bfbeta=0$ so $V(x)=\sqrt{\Psi_0(z)}$.
 We have $\partial_{z_j}V(x^*)=\frac{\alpha_j}{2\sqrt{\Psi_0(z^*)}}>0$ for $j=1,2,3$ so, by using\refq{val1.z-z*},
\begin{eqnarray*}
 \partial_{x_j}V(x) &=& \partial_{z_j}V(x) \partial_{x_j}z_j = \partial_{z_j}V(x) \sin(2x_j) \\
 &=& (\partial_{z_j}V(x^*) + O(z-z^*)) (\sin(2x_j^*) + 2\cos(2x_j^*)(x_j-x^*_j)\\
 && + O((x_j-x^*_j)^2)) \\
 &=& 2(\partial_{z_j}V(x^*) + O(z-z^*)) (s^*_j (x_j-x^*_j) + O((x_j-x^*_j)^2)) \\
 &=& C_j s^*_j (x_j-x^*_j) (1 + O(\rmd_0(x,x^*))),
\end{eqnarray*}
 where $C_j=2\partial_{z_j}V(x^*)=\frac{\alpha_j}{\sqrt{\Psi_0(z^*)}}> 0$. 
 Thus\refq{prop1.dV} is proved.\\
 Case 3-1 is similar since the six partial derivatives $\partial_{z_j}\tau^\pm$ are all constant
 and positive, and $\tau^\pm(z^*)>0$. (See\refq{val.tau+-anal} and\refq{val.tau--anal}.)
 \\
 Cases 2-1 and 2-2-a and 2-2-b are similar to Case 1, the sign of $C_j$ being a consequence
 of Lemma \ref{lem.dtau=0}.\\
 Let us be more precise in Case 2-2-a.
 Since $z^*_3=\nu\in (0,1)$ then $\sin(2x^*_3)\neq 0$ so, by using\refq{val0.z-z*},\refq{val1.z-z*},
 we have
\begin{eqnarray*}
 \partial_{x_3}V^2(x) &=& \partial_{z_3}\tau^-(z) \sin(2x_3) \\
 &=&  (O(z_1-z^*_1) + O(z_2-z^*_2) + \partial^2_{z_3}\tau^-(z^*)) (z_3-z^*_3)   \\
 &&  (1+ O((z_3-z^*_3))) (\sin(2x^*_3) + O(x_3-x^*_3)).
\end{eqnarray*}
 Since  $z_j-z^*_j=O((x_j-x^*_j)^2) = O(\partial_{x_j}\tau^-(z)\rmd_0(x,x^*))$ for $j=1,2$, and\\
 $z_3-z^*_3= \sin(2x^*_3) (x_3-x^*_3)(1+O(x_3-x^*_3))$ then 
\begin{eqnarray*}
 \partial_{x_3}V(x) &=& \frac{\partial_{x_3}V^2(x)}{2V(x)} = C_3 (x_3-x^*_3)  + O(d_0^2(x,x^*)) \\
 &=& C_3 (x_3-x^*_3) + |\nabla_x V(x)| O(\rmd_0(x,x^*)),
\end{eqnarray*}
 where $C_3= (2\sqrt{\tau^-(z^*)})^{-1}\partial^2_{z_3}\tau^-(z^*)\sin^2(2x_3^*)>0$. Thus\refq{prop1.dV} holds
 in Case 2-2-a too.\\
 Case 2-2-c. The computation of the derivatives $\partial_{x_j}V(x)$, $j=1,2$,
 is similar to the other cases (with $s_j=s^*_j=-1$).
 By using\refq{val1.z-z*} we have
\begin{eqnarray*}
 \partial_{x_3}V^2(x) &=& \partial_{z_3}\tau^-(z) \sin(2x_3) \\
 &=&
   \Big((O(z_1-z^*_1)+ O(z_2-z^*_2) +\partial^2_{z_3}\tau^-(z^*)) (z_3-z^*_3)+ O(z_3-z^*_3)^2 \Big)\\
&& \cdot \Big(2s^*_3 (x_3-x^*_3) + O((x_3-x^*_3)^3)\Big)   .
\end{eqnarray*}
 Hence
$$ \partial_{x_3}V(x) = C_3 (x_3-x^*_3)^3 + O(|\nabla_x V(x)|\, \rmd_0(x,x^*)), $$
 where $C_3= 2(\sqrt{\tau^-(z^*)})^{-1}\partial^2_{z_3}\tau^-(z^*)>0$. 
 Thus\refq{prop2.dV} holds too.
 
 The lemma is proved.
\subsection{Proof of Lemma \ref{lem.wdH}}
 We remember that $y^*_1\neq 0$, $z^*_2\neq 0$ and $z^*_3=0$, so,
$$  \sqrt{\Psi_0(z)}_w  = \cos(x_1)\cos(x_2) \tilde{\sqrt{\Psi_0(z)}}_{\tilde w}. $$
%
 Thanks to\refq{val.Tw} the function $x\mapsto \tilde{\sqrt{\Psi_0(z)}}_{\tilde w}=2\alpha_1z_1+2\alpha_2z_2$
 is smooth and positive in $\supp\chi^*_2$.\\
 In Case 3-3, we have $0<z^*_1<z^*_2=1$ then $\cos y^*_1\neq0$ so
 $\cos(x_1)=\cos(x^*_1)+O(\rmd_0(x,x^*))$ with $ \cos(x^*_1)\neq 0$, and
 $\cos(x_2)= -y^*_2(x_2-x^*_2)+O(x_2-x^*_2)^3$. Hence\refq{est1.xiw3_3} holds.\\
 In Case 3-2, we have $z^*_1=z^*_2$, $\cos y^*_1=0$ so
 $\cos(x_j)= -y^*_j(x_j-x^*_j)+O(x_j-x^*_j)^3$, $j=1,2$. Hence\refq{est1.xiw3_2} holds.
 
\subsection{Proof of Lemma \ref{lem1.Aadjoint}}
 We fix a representation of $x^*\in\T^3$ in $\R^3$ which we denote again $x^*$.
 Then, the multiplication by $\chi_{x^*}$ is an isometry (non surjective) from $\calH^D$
 into the Hilbert space $L^2_D(\R^3,\C^6):=L^2(\R^3,\C^6)$ equipped with the scalar product
$$
  (u,v)_{L^2_D(\R^3;\C^6)} := \int_{\R^3} <u(x), v(x)>_{\C^6,D} \rmd x
  = \int_{\R^3} <\hat D^{-1}  u(x), v(x)>_{\C^6} \rmd x.
$$
 So, we can identify $A_{x^*}$ with an unbounded symmetric operator on $L^2(\R^3;\C^6)$,
 which we denote $A_{x^*}$ again.
 
 We set $x'_j= \sqrt{C_j/2}(x_j-x^*_j)$ where the $C_j$'s are the positive constants of Section \ref{sec.p1},
 and in\refq{prop1.dV} or in\refq{prop2.dV} of Lemma \ref{lem.dV}.
 We set also $\rho'= \sqrt{x_1'^2+x_2'^2}$, $r'=\sqrt{\rho'^2+x_3'^2}$.

\medskip
 A-1)
 Let us consider Case 1 with $s_1=s_2=1$ and $s_3=-1$.
 Since $\bfbeta=0$ then we have $\bfmu=\kappa\bfepsl$, with the scalar $\kappa>0$. 
 We have
$$ p_1(x) = \rho'^2 - x_3'^2,$$
 and we set
%
$$
\begin{array}{rcl}
 p_2(x) &:=& 2 \rho' x'_3,\\
 p_3(x) &:=& \frac{(x_1',x_2')}{\rho'} \in \Sm^1 \approx \R/(2\pi\Z).
\end{array}
$$
%
 The mapping $\R^2\sauf\{(0,0)\} \ni  (x'_1,x'_2)\mapsto (\rho',p_3) \in \R^+\times\Sm^1$ is
 the polar change of coordinates.
 Since $p_1+ip_2= (\rho'+ix'_3)^2$, then the mapping $(\rho',x'_3) \mapsto (p_1,p_2)$ is
 a $\calC^\infty$-diffeomorphism from $(0,\infty)\times\R$ onto $\calO:= \R^2\sauf (\R^-\times\{0\})$.
 Thus the mapping
$$ \Phi: \: x'=(x'_1,x'_2,x'_3)\mapsto p=(p_1,p_2,p_3) $$
 is a $\calC^\infty$-diffeomorphism from $\R^{2*}\times\R$ onto $\calU:=\calO \times \Sm^1$,
 with jacobian
%
$$ J_\Phi(x')  =  - \frac{r'^2} {\rho'} . $$
%
 We set $\tilde H=L^2(\R^{2*}\times \Sm^1,\C^6; \rmd p)$ equipped with the following scalar product:
$$
 (\tilde u,\tilde v)_{\tilde H} := \int_{\R^2\times \Sm^1} <\tilde u(p), \tilde v(p)>_{\C^6,D} \rmd p.
$$
 For $\tilde u\in \tilde H$, $p\in\calU$, we set
\begin{equation}
\label{def.tildeu}
 u(x)= |J_\Phi(x')|^{1/2} \tilde u(p), \quad x'=\Phi^{-1}(p) \in \R^3,
\end{equation}
 so the transform
%
$$ T : \: L^2(\R^3,\C^6) \ni u \mapsto \tilde u \in  \tilde H $$
%
 is a bijective isometry.
 Setting $\tilde\pi(p)=\pi_2(x)$ and $\tilde\chi(p)=\chi_{x^*}(x)$, the partial derivatives
 $\partial^j_{p_1}\tilde\chi$, $j\ge 0$, are bounded in $\R^{2*}\times \Sm^1$ since
 $\tilde\chi=1$ near $p(x^*)=0$ and the function $|\nabla_x p_1|$ on $\supp \nabla\chi_{x^*}$
 is smooth and bounded by below by a positive constant.
 For example if $j=1$, we have 
$$
 \sup_{x\in B(\R^3)}
 |\partial^j_{p_1}\tilde\chi (p(x))| =
 \sup_{\frac{r}2\le \rmd_0(x,x^*)\le r} |\partial^j_{p_1}\tilde\chi (p(x))| \le C.
$$
 The projector $\tilde \pi$ is continuous but admits a singular of first order at $p=0$.
 Observing that
$$
  \frac{\nabla p_1(x)\nabla u(x)}{|\nabla p_1|^2} = \frac{\partial u}{\partial p_1} ,
$$
 and denoting by "$+$ sym." the terms of symmetrization of $\calA_{x^*}$,
 we have, for $u,v\in \calC^\infty_c((\R^2\sauf\{(x^*_1,x^*_2)\}\times\R)$, 
\begin{eqnarray*}
 (A_{x^*}u,v)_{\calHD} &=&  \int_{\R^3} i \chi_{x^*}(x) <\pi_2(y) 
 \frac{\nabla p_1(x)\nabla (\chi_{x^*}(x)\pi_2(y)u(x))}{|\nabla p_1(x)|^2},v(x)>_{\C^6,D} \rmd x + {sym.}\\
  &=&   \int_{\R^2\times \Sm^1} i <\frac{\partial (\tilde \pi \tilde u)}{\partial p_1},
   \tilde\chi^2(p) \tilde \pi(p) \tilde v(p)>_{\C^6,D}  \rmd p
    + {sym.}\\
 &=&   \int_{\R^2\times \Sm^1} i < \frac{\partial (\tilde\chi \tilde \pi \tilde u)}{\partial p_1},
  \tilde\chi \tilde \pi \tilde v>_{\C^6,D}  \rmd p
 \equiv (\tilde A_{x^*} \tilde u,\tilde v)_{\tilde H}.
\end{eqnarray*}
 The projection $\tilde\pi$ has range two.
 Since $\bfbeta=0$, we have, for $z\neq 0$, a basis of the eigenspace $\ker(H^D(x)-\sqrt{\Psi_0(z)})$
 of the form $(\varphi_1(p),\varphi_2(p)=\ove{\varphi_1(p)})^T$,
 with $\varphi_1=(q,i\sqrt{\kappa}q)$ and
 $q(p)^T\in \ker(i\sqrt{\kappa}\bfepsl M(y)-\Psi_0(z)I_3)$, where $x=x(p)$, $I_3$ denotes
 the identity matrix of size 3, and $M(y)$ is the $3\times 3$ matrix defined at\refq{def.matM}.
 Moreover we can choose $q(p)$ such that $x\mapsto q(p(x))$ is analytic in the support
 of $\chi_{x^*}$ at least, and with $<\bfepsl^{-1}q(p),q(p)>_{\C^3} = 1/2$,
 so $(\varphi_1(p),\varphi_2(p))$ is orthonormal in $\C^6$ equipped with $<,>_{\C^6,D}$. 
 We thus have $<\varphi_i(p),\varphi_j(p)>_{\C^6,D}=\delta_{i,j}$, $i,j\in\{1,2\}$, but also
\begin{eqnarray*}
 <\partial_p \varphi_1,\varphi_2>_{\C^6,D} &=&
  <\bfepsl^{-1}\partial_p q,q>_{\C^3} + <\bfmu^{-1}\partial_p(i\sqrt{\kappa}q), -i\sqrt{\kappa}q>_{\C^3} \\
 &=&   <\bfepsl^{-1}\partial_p q,q>_{\C^3} + <i\bfepsl^{-1}\partial_p q, -i q >_{\C^3}\\
 &=& 0.
\end{eqnarray*}
 Similarly, $<\varphi_1,\partial_p \varphi_2>_{\C^6,D}=0$.
 Hence we have
$$ \tilde\pi(p) \tilde u(p) = \tilde\xi_1(\tilde u)(p) \varphi_1(p)+\tilde\xi_2(\tilde u)(p) \varphi_2(p),$$ 
 where we set 
\begin{equation}
\label{def.xij}
 \tilde\xi_j(\tilde u)(p) := <\tilde u(p),\varphi_j(p)> _{\C^6,D} .
\end{equation}
 We then have
$$
 (\tilde A_{x^*}  \tilde u, \tilde v)_{\tilde H} =
   i \sum_{j=1}^2 \int_{\R^2\times \Sm^1}   \frac{\partial}{\partial p_1} (\tilde\chi\tilde\xi_j(\tilde u))
   \tilde\chi \ove{\tilde\xi_j(\tilde v)} \rmd p.
$$
 Let us set 
$$
 D(\tilde A_{x^*}) = \{\tilde u \in \tilde H; \; \tilde\chi^2 \partial_{p_1} \tilde\xi_j(\tilde u) \in
  L^2(\R^2\times\Sm^1,\C; \rmd p),\: j=1,2 \}. 
$$
 Let us show that $D(\tilde A^*_{x^*})=D(\tilde A_{x^*})$.
 Let $\tilde v\in D(\tilde A^*_{x^*})$, so we have:
\begin{equation}
\label{ine.DA*}
 |(\tilde A_{x^*} \tilde u,\tilde v)_{\tilde H}| \le C\|\tilde u\|_{\tilde H}, \quad  \forall \tilde u\in D(\tilde A_{x^*}) ,
\end{equation}
 that is,
%
$$
 |\sum_{j=1}^2 \int_{\R^2\times \Sm^1}  \frac{\partial \tilde\xi_j(\tilde u)}{\partial p_1} \tilde\chi^2(p)
  \ove{\tilde\xi_j(\tilde v)} \rmd p| 
 \le C\|\tilde u\|_{\tilde H}, \quad  \forall \tilde u\in D(\tilde A_{x^*}) .
$$
%
 We fix $j\in\{1,2\}$ and choose $\tilde u(p)=f(p_1)g(p_2,p_3)\varphi_j(p)$ in the above estimate
 with arbitrary $f\in H^1(\R;\C;\rmd p_1)$ and $g\in L^2(\R\times\Sm^1;\C;\rmd p_2\rmd p_3)$.
 Then $\|\tilde u\|_{\tilde H}\le C \|f\|_{H^1(\C)} \|g\|_{L^2(\R\times\Sm^1)}$ so we have
\begin{eqnarray*}
 |\int_{\R^2\times \Sm^1}  \frac{\partial f(p_1)}{\partial p_1} g(p_2,p_3) \tilde\chi^2(p)
  \ove{\tilde\xi_j(\tilde v)} \rmd p| 
 \le C\|f\|_{H^1(\R)} \|g\|_{L^2(\R\times\Sm^1)} , \\
 \quad  \forall f\in H^1(\R;\rmd p_1), \: g\in L^2(\R\times\Sm^1,\rmd p_2\rmd p_3).
\end{eqnarray*}
 It shows that
$$
 K(p_1) :=  \int_{\R\times \Sm^1}  \tilde\chi^2(p) \tilde\xi_j(\tilde v) g(p_2,p_3)
  \rmd p_2 \rmd p_3 \in H^1(\R;\C;\rmd p_1)
$$
 with
$$ \|  \frac{\partial}{\partial p_1} K(p_1)\|_{L^2(\R)} \le C\|g\|_{L^2(\R\times\Sm^1)} .$$
 But we have
\begin{eqnarray*}
 \frac{\partial}{\partial p_1} K(p_1) &=&  \int_{\R\times \Sm^1}  \tilde\chi^2(p) \frac{\partial}{\partial p_1} 
   \tilde\xi_j(\tilde v) g(p_2,p_3)  \rmd p_2 \rmd p_3 + L(p_1),\\
 L(p_1) &:=& \int_{\R\times \Sm^1} \tilde\xi_j(\tilde v) ( \frac{\partial}{\partial p_1} \tilde\chi^2(p))  g(p_2,p_3)  \rmd p_2 \rmd p_3,
\end{eqnarray*}
with $\|L_1\|_{L^2(\R)} \le C\|g\|_{L^2(\R\times\Sm^1)}$. Hence we have
$$
 \tilde\chi^2 \frac{\partial}{\partial p_1} \tilde\xi_j(\tilde v) \in L^2(\R^2\times\Sm^1,\C; \rmd p).
$$
 Thus, $\tilde v\in D(\tilde A_{x^*})$ and so $\tilde A_{x^*}$ is self-adjoint.
 Consequently, $A_{x^*}$ with domain $T^{-1}(D(\tilde A_{x^*}))$ is a self-adjoint operator.

\medskip
 Case 1 with the general situation $s_1s_2s_3=-1$ is similar.\\
 Cases 2-1 and 2-2-a and 2-2-b and 2-2-d and 2-2-d and 2-2-e, with $s_1s_2s_3=-1$,
 are similar, except that the projection $\tilde \pi$ ($=\pi_1^+(y)$ or $=\pi_1^-(y)$) has range one,
 which simplifies the proof.

\medskip
 Case 3-1 with $s_1s_2s_3=-1$.
 We set $\pi_1^\pm A_{x^*}\pi_1^\pm =: A^\pm$ so  $A_{x^*}=\sum_\pm A^\pm$
 with $D(A^\pm)=\calD_0$. We prove that $A^\pm$ is essentially self-adjoint on $\calHD$.
 We set
$$
 k^\pm(x) :=  
  \frac{|\nabla_x p_1(x)|} 
  {|\nabla_x p_1(x) \cdot \nabla_x \sqrt{\tau^\pm(z)}|^{1/2}}.
$$
 Thanks to Lemma \ref{lem.dV} we have 
$$ k^\pm(x) = 1+O(\rmd_0(x,x^*)). $$
 Thus $k^\pm(x)$ is defined for $x\simeq x^*$ and $x\neq x^*$, extends as
 a positive lipschitzian function near $x^*$.
 We then consider the same transforms than in Case 1 with $\tilde H^\pm$ replacing
 $\tilde H$ so we have
\begin{eqnarray*}
 (A^\pm u^\pm,v^\pm)_{\calHD} &=&
   i \int_{\R^2\times \Sm^1}  
   <\frac{\partial}{\partial p_1} (\tilde\chi \tilde u^\pm),
   \tilde\chi \ove{\tilde v^\pm}>_{\C^6,D}\; \rmd p\\
 &\equiv &  (\tilde A^\pm \tilde u^\pm, \tilde v^\pm)_{\tilde H} ,
\end{eqnarray*}
 where we set $\tilde A^\pm := i \tilde\chi \circ \frac{\partial}{\partial p_1} \circ \tilde\chi$,
 and
$$
  \tilde u^\pm(p) :=  |\nabla_x p_1(x)| |J_\Phi(x')|^{-1/2} \tilde k^\pm(p) u(x) 
  \quad x'=\Phi^{-1}(p) \in \R^3,
$$
 and $\tilde k^\pm(p):=k^\pm(x)$, $\tilde\pi_1^\pm(p):=\pi_1^\pm(y)$, $\tilde\chi(p):= \chi_{x^*}(x)$.
 Thus, as in Case 2-1 with $\pi_j s_j=1$, $A^\pm=\pi_1^\pm A^\pm \pi_1^\pm$
 is essentially self-adjoint on $\calHD$.
 We denote by $D^\pm$ the domain of the self-adjoint extension of $A^\pm$,
 so $D^\pm = \{u\in \calHD$; $A^\pm u\in \calHD\}$.
 Then, $A_{x^*}$ extends as a symmetric operator, $A'_{x^*}=A_{x^*}$ with domain
 $D(A'_{x^*}):=D^+\cap D^-$. 
 Now, let $v\in D((A'_{x^*})^*)$ so
$$ |(A_{x^*}u,v)_{\calHD}| \le C\|u\| \quad \forall u \in D(A'_{x^*}).$$
 Let $u\in D^\pm$. 
 Then $A_{x^*}\pi_1^\pm u=A^\pm u\in \calHD$, so $\pi_1^\pm u\in D(A'_{x^*})$.
 Thus
$$
   |(A^\pm u, v)_{\calHD}| =  |(A_{x^*} \pi_1^\pm u,v)_{\calHD}|
     \le C\|\pi_1^\pm u\|  \le C\|u\|  \quad \forall u \in D^\pm.
$$
 Hence $v\in D^\pm$. Thus, $v\in D(A'_{x^*})$, so $A'_{x^*}$ is self-adjoint.

\medskip
 Case 2-2-c.  We have
$$ p_1(x) = \rho'^2 - \frac12 x_3'^4,$$
 and we set
$$
\begin{array}{rcl}
 p_2(x) &:=& x_2' e^{-\frac1{2x_3'^2}},\\
 p_3(x) &:=& x_1' e^{-\frac1{2x_3'^2}},
\end{array}
$$
 with $p_2|_{x_3'=0}=p_3|_{x_3'}=0$, so $p_2,p_3\in \calC^\infty(\R^3)$.
 Then,
\begin{eqnarray*}
 \nabla p_1 &=& 2(x'_1,x'_2,-x_3'^3),\\
 \nabla p_2 &=& e^{-\frac1{2x_3'^2}} (0,1,x'_2/x_3'^3),\\
 \nabla p_3 &=& e^{-\frac1{2x_3'^2}} (1,0,x'_1/x_3'^3),
\end{eqnarray*}
 $\nabla p_1 \perp \nabla p_j$, $j=2,3$, and the Jacobian of the mapping
 $\Phi$: $x'\mapsto p$ is
$$
 J_\Phi(x') =   \frac{\rho'^2}{x_3'^3} e^{-\frac1{x_3'^2}}.
$$
%
 It does not vanish if $x'_3\neq 0$ or $\rho'\neq 0$.
 Let us invert $\Phi$. The sign of $x'_3$ is not determined by $p$ so we consider
$$ \Phi^\pm: \quad \R^{2*}\times\R^{\pm*} \ni x'\to p\in \R \times \R^{2*} .$$
 Let $p\in \R \times \R^{2*}$. We have $x'_1=p_3 e^{\frac1{2x_3'^2}}$, $x'_2=p_2 e^{\frac1{2x_3'^2}}$
 so $x'_3$ satisfies the equation $F(x_3'^2) = p_1$ where we set
$$ F(t) : =  (p_2^2 + p_3^2) e^{1/t} - \frac12 t^2 , \quad t>0.$$
 Since $F'>0$, $F(+\infty)=-\infty$ and $F(0^+)=+\infty$, then the equation is uniquely solvable
 by some $t_0>0$ so we obtain $x'_3=\pm\sqrt{t_0}\in\R^*$.
 Hence $\Phi^\pm$ is bijective.
 We let the lector to check that $\Phi^\pm$ is an homeomorphism from $\R^{2*}\times \R^{\pm*}$
 into $ \R \times \R^{2*}$.
 Hence, $\Phi^\pm$ is a $\calC^\infty$-diffeomorphism from $\R^{2*}\times \R^{\pm*}$ into $ \R \times \R^{2*}$\\
 We set the Hilbert spaces $\tilde H^\pm = L^2(\R^{2*}\times\R^{\pm*},\C^6; \rmd p)$ 
 equipped with the scalar product
$$
 (\tilde u^\pm,\tilde v^\pm)_{\tilde H^\pm} := \int_{\R^{2*}\times\R^{\pm*}}
  <\tilde u^\pm(p), \tilde v^\pm(p)>_{\C^6,D} \rmd p,
$$
 then $\tilde H := \tilde H^+ \oplus \tilde H^-$.
 For $\tilde u=(\tilde u^+,\tilde u^-) \in \tilde H$, $x'\in \R^{2*}\times\R^{\pm*}$, we set
%
$$ u(x)= |J_\Phi(x')|^{1/2} \tilde u^\pm(\Phi^\pm(x')), $$
%
 so the transform
$$ T : \:  L^2(\R^3) \ni u \mapsto \tilde u \in  \tilde H $$
 is a bijective isometry (up to a nonzero constant multiplicative factor).

 Setting again $\tilde\chi(p)=\chi_{x*}(x)$, $\tilde\pi(p)=\pi_2(x)$, we have,
 for $u,v\in \calC^\infty_c( \R^2\sauf\{(x^*_1,x^*_2)\}\times \R\sauf\{x^*_3\} )$, 
$$
 (A_{x^*}u,v)_{\calHD} =  \int_{\R^3} i < \frac{\partial (\tilde\chi \tilde \pi \tilde u)}{\partial p_1},
  \tilde\chi \tilde \pi \tilde v>_{\C^6,D}  \rmd p \equiv (\tilde A_{x^*} \tilde u,\tilde v)_{\tilde H}.
$$
 The projection $\tilde\pi$ has range one so this case is similar to case 2-1,
 so $A_{x^*}$ is essentially self-adjoint.

\medskip
  A-2) Let us treat Case 1 with $s_1=s_2=-1$ (and $s_3=-1$).
 We observe that
$$ p_1(x) = \sum_{j=1}^3 (x'_j)^2 \ge 0 . $$ 
%
 (where we set $x'_j= \sqrt{C_j/2}(x_j-x^*_j)$).
We use the spherical coordinates: 
 $x'=\rho\om$ with $\rho= \sqrt{{x'}_1^2 + {x'}_2^2 + {x'}_3^2}>0$, $\om=\rho^{-1}x'\in\Sm^2$,
 so we have $p_1=\rho^2$ and choose two other coordinates, $p_2,p_3$,
 on the sphere  $\Sm^2$.\\
 We then follow the above method (Case 1 with $s_1=s_2=1=-s_3$) with similar notations,
 notably, with the same couple $(\varphi_1,\varphi_2)$ and coordinates $\tilde\xi_j$
 (defined by\refq{def.xij}) $j=1,2$.
 The mapping
$$ \Phi: \: x'=(x'_1,x'_2,x'_3)\mapsto p=(p_1,p_2,p_3) $$
 is a $\calC^\infty$-diffeomorphism from $\R^{3*}$ onto $\R^{+*}\times\Sm^2$.
 The jacobian of $\Phi$ has the form
$$ J_\Phi(x')  =  j(p_2,p_3)\sqrt{p_1} , $$
 where $j$ is a positive smooth function on $\Sm^2$.
 We set $\tilde H=L^2(\R^{+*}\times \Sm^2,\C^6; \rmd p)$ equipped with the following scalar product:
$$
 (\tilde u,\tilde v)_{\tilde H} := \int_{\R^{+*}\times \Sm^2} <\tilde u(p), \tilde v(p)>_{\C^6,D} \rmd p.
$$
 For $\tilde u\in \tilde H$, $p\in\calU$, we consider the transformation defined by\refq{def.tildeu}  
 between $u$ and $\tilde u$ so it is a bijective isometry (up to a positive constant multiplicative factor)
 between $L^2(\R^3,\C^6)$ and $\tilde H$ which we denote $T$ again.
 Setting $\tilde \pi(p)=\pi_2(y)$, we have $\tilde\chi\in \calC^\infty_c(\R^{+*}\times \Sm^2)$
 and $\tilde\chi=1$ near $p(x^*)=0$.
 We thus have, for $u,v\in \calC^\infty_c(\R^{+*}\times \Sm^2,\C^6)$, 
$$
 (A_{x^*}u,v)_{\calHD} = \sum_{j=1}^2 \int_{\R^{+*}\times \Sm^2}  i\tilde\chi 
   \frac{\partial (\tilde\chi \tilde\xi_j(\tilde u))}{\partial p_1}
   \ove{\tilde\xi_j(\tilde v)} \rmd p  \equiv (\tilde A_{x^*} \tilde u,\tilde v)_{\tilde H}.
$$
 The above formula defines the symmetric operator $\tilde A_{x^*}$ on $\tilde H$ with domain
 $\calC^\infty_c(\R^{+*}\times \Sm^2)$.
 Thus, $\tilde A_{x^*}$ extends to the operator with the same formula defined on
\begin{eqnarray*}
 D(\tilde A_{x^*}) &=& \calH_{1,0} := \{\tilde u \in \tilde H;\; \tilde\chi^2\partial_{p_1}\tilde\xi_j(\tilde u)
  \in L^2(\R^{+*}\times \Sm^2,\C; \rmd p),\\
 &&  \; \tilde\chi^2\partial_{p_1}\tilde\xi_j(\tilde u)|_{p_1=0}=0,\; j=1,2\}.
\end{eqnarray*}
 Let us prove that the default index $N^+$ of $\tilde A_{x^*}$ vanishes.
 Firstly, observe that
\begin{equation}
\label{val1.DtildeA*}
 D((\tilde A_{x^*})^*) = \calH_1 := \{\tilde u \in \tilde H;\; \tilde\chi^2\partial_{p_1}\tilde\xi_j(\tilde u)
  \in L^2(\R^{+*}\times \Sm^2,\C; \rmd p), \; j=1,2\}.
\end{equation}
 In fact an integration by parts shows that $D((\tilde A_{x^*})^*)$ contains $ \calH_1$.
 Then, let $\tilde v\in D((\tilde A_{x^*})^*)$ so\refq{ine.DA*} holds.
 As in Case 1 with $\Pi_{k=1}^3 s_k=-1$, let $j\in \{1,2\}$ and choose
 $\tilde u(p)=f(p_1)g(p_2,p_3)\varphi_j(p)$ with arbitrary $f\in H^1(\R^{+*};\C;\rmd p_1)$
 and $g\in L^2(\Sm^2;\C;\rmd p_2\rmd p_3)$, so we have
\begin{eqnarray*}
 |\int_{\R^+\times \Sm^2}  \frac{\partial f(p_1)}{\partial p_1} g(p_2,p_3) \tilde\chi^2(p)
  \ove{\tilde\xi_j(\tilde v)} \rmd p| 
 \le C\|f\|_{H^1(\R^+)} \|g\|_{L^2(\Sm^2)} , \\
 \quad  \forall f\in H^1(\R;\rmd p_1), \: g\in L^2(\Sm^2,\rmd p_2 \rmd p_3).
\end{eqnarray*}
 It implies $\frac{\partial }{\partial p_1} (\tilde\chi^2 \tilde\xi_j(\tilde v))\in L^2(\R^{+*}\times \Sm^2,\C; \rmd p)$,
 then $\tilde\chi^2 \frac{\partial }{\partial p_1} \tilde\xi(\tilde v)\in $ $L^2(\R^{+*}\times \Sm^2,\C; \rmd p)$,
 so $\tilde v\in \calH_1$. Therefore,\refq{val1.DtildeA*} is proved.
 Now, let $\tilde v\in D((\tilde A_{x^*})^*)$ such that $(\tilde A_{x^*})^* \tilde v=i\tilde v$. 
 Thus we have $(-i(\tilde A_{x^*})^* \tilde v,\tilde v)_{\tilde H}= (\tilde v,\tilde v)_{\tilde H}$.
 An integration by parts (according to the variable $p_1$) shows that $v=0$.
 Consequently, $A_{x^*}$ with domain $T^{-1}(D(\tilde A_{x^*}))$ is a maximal symmetric operator
 with the default index $N^+=0$. (See also\cite[Lemma 1.3]{GEO.COM} for results of the same kind).



\medskip
 C) (Cases 3-2 and 3-3).
 We set 
$$ p_1 = \cos(x_1)\cos(x_2), $$
 so $p_1$ vanishes at $x=x^* \in \calX_2^*$ (such that $z^*=(\frac{\beta_2}{\beta_1},1,0)$).
 We have 
$$ \nabla_x p_1 = -w = (\sin x_1 \cos x_2, \sin x_2 \cos x_1,0).$$
 We set
$$ p_2 = \frac{\sin x_1}{\sin x_2}- \frac{\sin x^*_1}{\sin x^*_2}, \quad p_3=x_3,$$
 so $\nabla_x p_i\cdot \nabla_x p_j=0$ if $i\neq j$ and the jacobian of the map
 $\Phi$: $x\mapsto p=(p_1,p_2,p_3)$ is
$$ J_\Phi(x) = \Pi_{j=1}^3 \nabla_x p_j .$$
 We have $w\cdot \nabla_x u=u_w = -|\nabla p_1|^2 \partial_{p_1}u$ and, thanks to\refq{val.fw},
 $\sqrt{\Psi_0(z)}_w = p_1\tilde{\sqrt{\Psi_0(z)}}_{\tilde w}$ where $\tilde{\sqrt{\Psi_0(z)}}_{\tilde w}$
 is analytic and does not vanishes at $x^*$.

 Case 3-3. We have $\nabla_x p_1(x^*)\neq 0$ and $J_\Phi(x^*)\neq 0$, so $\Phi$
 is a local diffeomorphism from a neighborhood of $x^*$ in $\R^3$ into a neighborhood
 of $0_{\R^3}$ in $\R^3$.
 Hence, we have
$$
 (A_{x^*}u,v)_{\calHD} = -i  \int_{\R^3} \frac{k(x)}{p_1} < \frac{\partial(\chi_2 \pi_2 u)}{\partial p_1}, 
  \chi_2 \pi_2 v>_{\C^6,D} \rmd p + {\rm sym},
$$
 where $k$ is smooth with $k(x^*)>0$.
 As in the above cases, we thus set $q=(q_1,q_2,q_3)$, $q_1= p_1|p_1|$, $q_j=p_j$ for $j=2,3$, 
 $\tilde \chi(q)=\chi_{x^*}(x)$, $\tilde\pi(q)=\pi_2(y)$, $\tilde u(q) = |k(x)|^{1/2} u(x)$.
 We then obtain
$$
 (A_{x^*}u,v)_{\calHD}
 = -i \int_{\R^3} \sgn(q_1) < \frac{\partial(\tilde \chi \tilde \pi \tilde u )}{\partial q_1}, 
  \tilde \chi_2 \tilde \pi_2 \tilde v>_{\C^6,D} \rmd q
 \equiv (\tilde A_{x^*} \tilde u,\tilde v)_{\tilde H},
$$
 where $\tilde H = L^2(\R^3,(\C^6,<,>_{\C^6,D}); \rmd q)$ is a usual Hilbert space.
 The projection $\tilde\pi$ has range two so Case 3-3 is similar to A-2) with
 $A_{x^*}$ replaced by $-A_{x^*}$.
 Hence, $A_{x^*}$ has default index $N^-=0$ and admits a maximal symmetric extension.

\medskip
 Case 3-2. We have $\nabla_x p_1(x^*)= 0$ so $J_\Phi(x^*)= 0$.
 Let us "invert" $x\mapsto p$. For simplicity we assume $y^*_1=y^*_2=1$.
 Set $x'_j=x_j-x^*_j$ for $j=1,2$.
 Since $\sin x_j\simeq 1- (x'_j)^2/2$ and $\cos(x_j)\simeq -x'_j$ for $j=1,2$ then
 $p_1\simeq x'_1x'_2 $ and $-2p_2\simeq (x'_1)^2 -(x'_2)^2 $.
 Thus $(x'_1+ix'_2)^2\simeq 2i(p_1+ip_2)$.\\
 It means that we have the same transform than in Case A-1), i.e., there exists
 an Hilbert space $\tilde H$ and an isometry $L^2(\R^3)\ni u\to \tilde u \in \tilde H$ such that
$$
 (A_{x^*}u,v)_{\calHD} = -i \int_{\R^2\times\Sm^1} \sgn(q_1)
  < \frac{\partial(\tilde \chi \tilde \pi \tilde u )}{\partial q_1}, 
  \tilde \chi_2 \tilde \pi_2 \tilde v>_{\C^6,D} \rmd q,
$$
 where $\tilde\chi(q):=\chi_{x^*}(x)$, $\tilde\pi(q):=\pi_2(y)$.
 Hence, as in Case 3-3, $A_{x^*}$ has default index $N^-=0)$ and
 admits a maximal symmetric extension.

\subsection{Notations}
\label{sec.notation}
\begin{eqnarray*}
 y &=& \sin x \quad (y_j=\sin x_j), \hspace{4cm} \\
 z &=& y^2 \quad (z_j=y_j^2),\\
 \bfbeta &=& \bfepsl \times \bfmu = (\beta_1,\beta_2,\beta_3),\\
 \bfalpha &=& (\alpha_1,\alpha_2,\alpha_3),\quad
  \alpha_1 := (\epsl_2  \mu_3 + \epsl_3 \mu_2)/2 \quad \mbox{and c.p.},\\
 \gamma_1 &=&  \epsl_2 \epsl_3 \mu_2\mu_3  \quad \mbox{and c.p.},\\
 \nu &=&  \frac{2\alpha_3\sqrt{\beta_1\beta_2} - \sqrt{\gamma_3}(\beta_1 + \beta_2)}
 {|\beta_3| \sqrt{\gamma_3}},\\
\Phi_0 &=& 
  \epsl_2 \epsl_3  \mu_2\mu_3 z_1^2  + (\epsl_2 \epsl_3 \mu_1\mu_3
   + \epsl_1 \epsl_3 \mu_2\mu_3)z_1 z_2 + {\rm c.p.},\\
  K_0 &=&    \frac14\left( \beta_1^2 z_1^2 + \beta_2^2 z_2^2 
   + \beta_3^2 z_3^2 - 2\beta_1\beta_2 z_1z_2 
   - 2\beta_2\beta_3 z_2z_3 - 2\beta_1\beta_3 z_1z_3 \right),\\
\quad \Psi_0 &=& \bfalpha\cdot z \: = \,  \alpha_1 z_1 + \alpha_2 z_2 + \alpha_3 z_3 ,\\
 \tau^\pm &=&  \Psi_0\pm\sqrt{K_0} ,\\
 \textcolor{red}{\lambda_\pm} &=& \textcolor{red}{\lambda_+\max\{\sqrt{\tau^\pm(z)} }| \; z \in [0,1]^3\},\\
 P_M &:& \R\times\T^3 \ni (\lambda,x) \mapsto x\in \T^3,\\
 P_\R &:& \R\times\T^3 \ni (\lambda,x) \mapsto \lambda\in \R,\\
 \Sigma &=& \{(\lambda,x)\,| \; \lambda\in \sigma(H^D(x))\}  = \cup_{j=1}^6 \Sigma_j,\\
 \calX_j &= & P_M(\Sigma_j),\quad j=1,2,\\
 \Sigma_1 &=& \{(\lambda,x)\in \Sigma; \: K_0(z)\neq0\: \},\\
 \Sigma_1^{*\pm} &=& \{(\lambda,x)\in \Sigma_1; \: \lambda^2=\tau^\pm(z), \nabla_x\tau^\pm(z)=0\},\\
 \Sigma_1^* &=&  \Sigma_1^{*+} \cup  \Sigma_1^{*-},\\
 \Sigma_2 &=& \{(\lambda,x)\in \Sigma; \: z\neq 0 ,\; K_0(z)=0, \: \lambda^2=\Psi_0(z)\},\\
 \Sigma_2^* &=& \{(\lambda,x)\in \Sigma_2; \; \mbox{$\nabla_x\Psi_0(z)$ is normal to $\calX_2$ at $x$}\},\\
 \calX_0  &=& \{x\in \T^3; \; z=0\}, \\
 \T_0^3 &=& \T^3 \sauf \calX_0,\\
 Z_{\{0,1\}} &=& \{0,1\}^3,\\
 Z_{\{0,1\}}^* &=& Z_{\{0,1\}} \sauf \{(0,0,0)\}, \\
 X_{\{0,1\}} &=& \{ x\in \T^3; \; z\in Z_{\{0,1\}}\},\\
 X_{\{0,1\}}^* &=& \{ x\in \T^3; \; z\in Z_{\{0,1\}}^*\},\\
 \calT &=&  \calT_1 \cup  \calT_2 \cup \{0\},\\
 \calT_j &=& P_\R(\Sigma_j^*), \\
 \calT_1^\pm &=& P_\R(\Sigma_1^{*\pm}),\\
 \calX_j^* &=& P_M(\Sigma_j^*),\\
 \calX_1^{*\pm} &=& P_M(\Sigma_1^{*\pm}),\\
 \calX^* &=& \calX_1^* \cup \calX_2^*,\\
 \calZ_1^{*\pm} &=& \sin^2(\calX_1^\pm),\\
 \sin^2(\calX^*)  &=&  \sin^2(\calX^*_1) \cup  \sin^2(\calX^*_2),\\
 A_{out} &=& \calA_{out} + \calA_{out}^*,\\
 A_\phi &=& A_{out} + A_{in} = A_0 + A_1 + A_2,\\
 A_0 &=& A_{out}.
\end{eqnarray*}

\end{document}